\documentclass[10pt,oneside,reqno]{amsart}

\usepackage{amsmath,amsthm} % Redundant if using the amsart or amsbook document classes. 
\usepackage{mathtools}      % An extension of amsmath (technically including amsmath before it is redundant)
\usepackage{enumerate}
\numberwithin{equation}{section}
\theoremstyle{definition}
\newtheorem{Definition}{Definition}[section]
\newtheorem{Example}[Definition]{Example}
\newtheorem{Remark}[Definition]{Remark}
\newtheorem{Assumption}[Definition]{Assumption}
\newtheorem{Problem}[Definition]{Problem}
\theoremstyle{plain}
\newtheorem{Theorem}[Definition]{Theorem}
\newtheorem{MainTheorem}{Theorem}

\newtheorem{Proposition}[Definition]{Proposition}
\newtheorem{Corollary}[Definition]{Corollary}
\newtheorem{Lemma}[Definition]{Lemma}

\usepackage{amssymb}        % This also loads amsfonts.
\usepackage{mathrsfs}       % \mathscr will use the Ralph Smith’s For­mal Script.
\usepackage{eucal}          % \mathcal will use the Euler Script.
\usepackage{stmaryrd}       % St Mary’s Road symbol font. Contains \llbracket, \rrbracket, \mapsfrom.

\usepackage{geometry}
\usepackage[labelfont=rm]{subcaption}

\usepackage{hyperref}

\usepackage[usenames,dvipsnames]{xcolor}
\usepackage{tikz}
\usetikzlibrary{arrows, matrix}
\usepackage{tikz-cd}
\usepackage{xifthen}        % Better conditionals when drawing

\newcommand{\al}{\alpha}

\newcommand{\ga}{\gamma}
\newcommand{\Ga}{\Gamma}

\newcommand{\ep}{\varepsilon}

\newcommand{\la}{\lambda}

\newcommand{\si}{\sigma}

\newcommand{\om}{\omega}

\newcommand{\N}{\mathbb{N}}
\newcommand{\Z}{\mathbb{Z}}

\newcommand{\R}{\mathbb{R}}
\newcommand{\C}{\mathbb{C}}
\newcommand{\K}{\Bbbk}

\newcommand{\T}{\mathbb{T}}

\newcommand{\Be}{\mathbf{e}}

\newcommand{\Bt}{\boldsymbol{t}}
\newcommand{\Bsi}{\boldsymbol{\sigma}}

\newcommand{\Fsl}{\mathfrak{sl}}

\newcommand{\Fg}{\mathfrak{g}}
\newcommand{\Fh}{\mathfrak{h}}

\newcommand{\Fm}{\mathfrak{m}}
\newcommand{\Fn}{\mathfrak{n}}

\newcommand{\CA}{\mathcal{A}}
\newcommand{\CB}{\mathcal{B}}
\newcommand{\CC}{\mathcal{C}}
\newcommand{\CI}{\mathcal{I}}
\newcommand{\CJ}{\mathcal{J}}

\newcommand{\CT}{\mathcal{T}}

\DeclareMathOperator{\ad}{ad}

\DeclareMathOperator{\Ann}{Ann}
\DeclareMathOperator{\Aut}{Aut}

\DeclareMathOperator{\Id}{Id}

\DeclareMathOperator{\ord}{ord}

\DeclareMathOperator{\sgn}{sgn}

\DeclareMathOperator{\Specm}{Specm}
\DeclareMathOperator{\Stab}{Stab}
\DeclareMathOperator{\Supp}{Supp}

\newcommand{\iv}[2]{\llbracket #1,#2 \rrbracket}
\newcommand{\un}{\underline}

\renewcommand{\tilde}{\widetilde}

\newcommand{\setupcylinder}[3]{
  \def\r{2}       % Physical radius of cylinder
  \def\h{5}       % Physical height of cylinder
  
  \def\m{#1}      % Angular part of the period
  \def\n{#2}      % Vertical part of the period
  \def\vh{#3}     % Vertical height (number of vertical grid steps to draw)
  
  \def\res{12}     % Resolution (sub-steps per angular grid step)
  
  \def\d{(360/\m)/\res}         % Angular sub-step size (angular part)
  \def\s{-(\h/\vh)*(\n/\m)/\res} % Angular sub-step size (vertical part)
  \def\v{\h/\vh}                 % Vertical step size

  \draw[-,Gray,very thin] (\r,0,0)
    \foreach \t in {5,10,...,360}
      {--({\r*cos(\t)},{\r*sin(\t)},0)};
  \draw[-,Gray,very thin] (\r,0,\vh*\v)
    \foreach \t in {5,10,...,360}
      {--({\r*cos(\t)},{\r*sin(\t)},\vh*\v)};

  \foreach \t in {1,...,\m}
  {
    \def\aa{(\t-1)*360/\m}
    \draw[-,Gray,very thin] ({\r*cos(\aa)},{\r*sin(\aa)},0) to ({\r*cos(\aa)},{\r*sin(\aa)},\vh*\v);
  }
  
  \foreach \t in {-\n,...,\vh}
  {
    \drawlatticepath{0}{\n+\t}{0,0,0,0,0,0,0}{-,Gray,very thin}
  }
 
}

\newcommand{\drawlatticestep}[4]{
  \def\p{#1*\d*\res}
  \def\z{#1*\s*\res + #2*\v} 

  \ifthenelse{#3 = 0}
  {
    \foreach \t in {1,2,...,\res}
    {
      \def\zz{\z+\t*\s}
      \pgfmathifthenelse{(\zz<=\h)*(\zz>=0)}{"\noexpand \draw[#4]
          ({\r*cos(\p+(\t-1)*\d)},{\r*sin(\p+(\t-1)*\d)},{\z+(\t-1)*\s})--
          ({\r*cos(\p+\t*\d)},{\r*sin(\p+\t*\d)},{\zz});"}{}\pgfmathresult
    }
  }
  {
    \draw[#4] ({\r*cos(\p)},{\r*sin(\p)},{\z})--
                          ({\r*cos(\p)},{\r*sin(\p)},{\z+\v});
  }
}

\newcommand{\drawlatticedot}[2]{
  \def\p{#1*\d*\res}
  \def\z{#1*\s*\res + #2*\v} 
  \fill ({\r*cos(\p)},{\r*sin(\p)},{\z}) circle (2pt);
}

\newcounter{x}
\newcounter{y}
\newcommand{\drawlatticepath}[4]{
  \setcounter{x}{#1}
  \setcounter{y}{#2}
  \foreach \i in {#3} {
    \drawlatticestep{\value{x}}{\value{y}}{\i}{#4}
    \addtocounter{x}{1-\i}
    \addtocounter{y}{\i}
  }
}

\title{Noncommutative fiber products and lattice models}
\author{Jonas T. Hartwig}
\address{Department of Mathematics, Iowa State University, Ames, IA 50014}
\email{jth@iastate.edu}
\urladdr{http://jth.pw}
\date{}

\begin{document}
\maketitle
\begin{abstract}
We establish a connection between the representation theory of certain noncommutative singular varieties and two-dimensional lattice models. Specifically, we consider noncommutative biparametric deformations of the fiber product of two Kleinian singularities of type $A$.
Special examples are closely related to Lie-Heisenberg algebras, the affine Lie algebra $A_1^{(1)}$, and a finite W-algebra associated to $\Fsl_4$.

The algebras depend on two scalars and two polynomials that must satisfy the Mazorchuk-Turowska Equation (MTE), which we re-interpret as a quantization of the ice rule (local current conservation) in statistical mechanics. Solutions to the MTE, previously classified by the author and D. Rosso, can accordingly be expressed in terms of multisets of higher spin vertex configurations on a twisted cylinder.

We first reduce the problem of describing the category of weight modules to the case of a single configuration $\mathscr{L}$. Secondly, we classify all simple weight modules over the corresponding algebras $\mathcal{A}(\mathscr{L})$, in terms of the connected components of the cylinder minus $\mathscr{L}$.
Lastly, we prove that $\mathcal{A}(\mathscr{L})$ are crystalline graded rings (as defined by Nauwelaerts and Van Oystaeyen), and describe the center of $\mathcal{A}(\mathscr{L})$ explicitly in terms of $\mathscr{L}$.

Along the way we prove several new results about twisted generalized Weyl algebras and their simple weight modules.
\end{abstract}

\section{Introduction}
In this section we motivate and define the main objects $\CA_{\al_1,\al_2}(p_1,p_2)$ and their connection with lattice models. The three main results of the paper are stated in Section \ref{sec:main-results}. In Section \ref{sec:tgwa} we give some background on the more general framework of twisted generalized Weyl algebras and also prove new results. The following three sections are dedicated to the proofs of the main theorems. Lastly, in Section \ref{sec:conclusion}, we reflect on how some of the outcomes relates to statistical mechanics and percolation, and state some open problems. A short appendix contains details about Example \ref{ex:finite-W}.

\subsection*{Notation and terminology}
The set of non-negative integers is denoted by $\N$. 
All rings $R$ are assumed to have a multiplicative identity $1=1_R$, and ring homomorphisms $R\to S$ assumed to map $1_R\mapsto 1_S$. By a \emph{regular} element of a ring we mean an element which is not a zero-divisor.
Without modifier, ``ideal'' means two-sided ideal, and ``module'' means left module.

\subsection*{Acknowledgements}
The author is grateful for interesting discussions with Andrew Linshaw, Tomoyuki Arakawa, Michael Damron, Christoffer Hoffman and Bernard Lidick\'{y}.

\subsection{Noncommutative Kleinian singularities}
One of the most investigated objects in the field of noncommutative algebraic geometry are the \emph{noncommutative type $A$ Kleinian singularities} $\CA_\hbar(f)$. These algebras were introduced by Hodges \cite{Hod1993}, studied as special cases of generalized Weyl algebras by Bavula \cite{Bav1991}, and are also known as polynomial Heisenberg algebras in the physics literature \cite{CarFer2004}. 
They have been studied intensively from many points of view
\cite{Bav1991,Bav1992,Bav1996,Jor1993,BavJor2001,Brz2016}.
A uniform generalization to any $ADE$ type was given in \cite{CraHol1998}, and type $D$ was studied separately in \cite{Bod2006,Lev2009}.

They are defined as follows. Let $\hbar\in\C$ be a deformation parameter and $f$ a nonzero polynomial. Then $\CA_\hbar(f)$ is the associative algebra generated by $\{X^+, X^-, H\}$ subject to defining relations
\begin{subequations}\label{eq:NCKS-rels}
\begin{alignat}{4}
 HX^+ - X^+H &= \hbar X^+ &\quad\quad\quad X^+X^- &= f\big(H-\tfrac{\hbar}{2}\big) \\
 HX^- - X^-H &=-\hbar X^- &\quad\quad\quad X^-X^+ &= f\big(H+\tfrac{\hbar}{2}\big)
\end{alignat}
\end{subequations}
If all zeros of $f$ belong to a single coset in $\C$ modulo $\Z\hbar$, then letting $\hbar\to 0$ we get
\[\CA_0(f)\simeq \C[x,y,z]/(xy-z^n),\quad n=\deg f,\]
which is the algebra of functions on the Kleinian singularity of type $A_{n-1}$.
If $\hbar\neq 0$ then, after rescaling $H$, one can assume $\hbar=1$.

\subsection{Noncommutative Kleinian fiber products}
The purpose of this paper is to study the following rank two generalization of noncommutative Kleinian singularities.

\begin{Definition}
\label{def:NCKFP}
Let $(\al_1,\al_2)\in\C^2$ and $(p_1,p_2)\in(\C[u]\setminus\{0\})^2$ where $u$ is an indeterminate.
Let $\tilde{\CA}=\tilde{\CA}_{\al_1,\al_2}(p_1,p_2)$ be the associative algebra with generators
$\{H,X_1^+,X_1^-,X_2^+,X_2^-\}$ subject to the following defining relations:
\begin{subequations}\label{eq:Aalbepq-rels}
\begin{alignat}{6}
HX_i^+ - X_i^+ H &= \al_i X_i^+ &\qquad X_i^+ X_i^- &= p_i\big(H-\tfrac{\al_i}{2}\big) &\qquad X_1^+ X_2^- &= X_2^- X_1^+ \\ 
HX_i^- - X_i^- H &=-\al_i X_i^- &\qquad X_i^- X_i^+ &= p_i\big(H+\tfrac{\al_i}{2}\big) &\qquad X_1^- X_2^+ &= X_2^+ X_1^-
\end{alignat}
\end{subequations}
The corresponding \emph{noncommutative (type $A\times A$) Kleinian fiber product} is defined as $\CA=\CA_{\al_1,\al_2}(p_1,p_2)=\tilde{\CA}/\CI$, where $\CI$ is the ideal consisting of all $a\in\tilde{\CA}$ such that $f(H)\cdot a=0$ for some nonzero polynomial $f\in\C[u]$. 
\end{Definition}

The are several reasons for taking the quotient by the ideal $\CI$. One reason is that $\tilde{\CA}$ is in general not a domain, but $\CA$ is. In fact, it can be shown that $\CI$ is the unique minimal completely prime ideal trivially intersecting $\C[H]$.
Another reason is that $\CI$ takes care of the missing relations between $X_1^+$ and $X_2^+$ (and between $X_1^-$ and $X_2^-$) in the following way. In $\tilde{\CA}$ one can deduce relations like
(see proof of Proposition \ref{prp:NKFP-Consistency})
\[ X_1^+X_2^+ p_1\big(H+\tfrac{\al_1}{2}\big) = X_2^+X_1^+ p_1\big(H+\tfrac{\al_1}{2}+\al_2\big).\]
If the polynomials in the right hand sides have common factors we want to cancel those, and this is allowed in the quotient $\CA$. For example, if $\al_2=0$ then $X_1^+$ and $X_2^+$ commute in $\CA$.

\begin{Example}[Noncommutative Kleinian singularities]
If $\al_2=0$ and $p_2(u)=1$ (constant) then $\CA_{\al_1,0}(p_1,1)\simeq \CA_{\al_1}(p_1)[Z,Z^{-1}]$,
 a central extension of a noncommutative Kleinian singularity.
\end{Example}

\begin{Example}[Commutative limit] Suppose that all zeros of $p_1$ and $p_2$ belong to a single coset in $\C$ modulo $\Z\al_1+\Z\al_2$. Then, as $\al_1,\al_2\to 0$, the corresponding noncommutative Kleinian fiber product $\CA_{\al_1,\al_2}(p_1,p_2)$ becomes isomorphic to the commutative algebra
\[ \mathscr{O}(X_{n,m})=\C[x_1,x_2,y_1,y_2,z]/\big(x_1y_1-z^n,\, x_2y_2-z^m\big), \]
where $n=\deg p_1$ and $m=\deg p_2$. This is the algebra of regular functions on a fiber product of two type $A$ Kleinian singularities along the subvariety $z_1=z_2$.
Thus the algebras $\CA_{\al_1,\al_2}(p_1,p_2)$ are noncommutative biparametric deformations of $\mathscr{O}(X_{n,m})$.
\end{Example}

\begin{Example}[Skew group algebra]
If $p_1$ and $p_2$ are both constant polynomials, then
$\CA_{\al_1,\al_2}(p_1,p_2)$ is isomorphic to the skew group algebra $\C[H]\ast_\si \Z^2$,
where $\si:\Z^2\to\Aut_\C\big(\C[H]\big)$ is given by $\si(k_1,k_2)\big(f(H)\big) = f(H-k_1\al_1-k_2\al_2)$.
\end{Example}

\begin{Example}[Twisted generalized Weyl algebras]
The class of \emph{twisted generalized Weyl algebras (TGWAs)} was introduced in \cite{MazTur1999} and further studied in \cite{MazPonTur2003,Har2006,HarOin2013}.
Any noncommutative Kleinian fiber product is a TGWA of rank two. Conversely, any TGWA of rank two with base ring $R=\C[u]$ and automorphisms $\si_i(u)=u-\al_i$ for $i=1,2$ is a noncommutative Kleinian fiber product (see Corollary \ref{cor:fiber-TGWA}). 
On the other hand, $\mathcal{A}_{\al_1,\al_2}(p_1,p_2)$ is graded isomorphic to a rank two \emph{generalized Weyl algebra} in the sense of Bavula \cite{Bav1992} if and only if $0\in\{\al_1,\, \al_2,\, \deg p_1,\, \deg p_2\}$.
\end{Example}

\begin{Example}[Type $A_2$]
Let $(\al_1,\al_2)=(-1,1)$ and $p_1(u)=p_2(u)=u$.
Then $\CA_{\al_1,\al_2}(p_1,p_2)$ is isomorphic to the twisted generalized Weyl algebra of Lie type $A_2$ given in \cite[Sec. 2, Ex.~3]{MazTur1999}. A complete presentation for this algebra involving Serre relations of type $A_2$ was given in \cite[Ex.~6.3]{Har2009}.
\end{Example}

\begin{Example}[Crystalline graded rings]
Any noncommutative Kleinian fiber product $\CA_{\al_1,\al_2}(p_1,p_2)$ is non-trivially a \emph{crystalline graded ring} \cite{NauVan2008} (see Corollary \ref{cor:PIDCGR} and Example \ref{ex:PIDCGR}).
\end{Example}

\begin{Example}[Affine Lie algebra of type $A_1^{(1)}$]
\label{ex:affine-A11}
Let $(a_{ij})=\left[\begin{smallmatrix}2&-2\\-2&2\end{smallmatrix}\right]$ and
let $U(\Fg)$ be the enveloping algebra of the Lie algebra generated by $e_i,f_i,h_i$, $i\in\{1,2\}$ subject to
\begin{subequations}
\begin{gather}
[e_i,f_j]=\delta_{ij} h_i,\qquad [h_i,e_j]=a_{ij} e_j,\qquad [h_i,f_j]=-a_{ij} f_j,\\
[e_i,[e_i,[e_i,e_j]]]=0,\qquad [f_i,[f_i,[f_i,f_j]]]=0,\qquad i\neq j.
\end{gather}
\end{subequations}
Let $(\al_1,\al_2)=(-1,1)$, and for $d\in \N$ put $p^d_1(u)=p^d_2(u)=\big(u-\frac{1}{2}\big)\big(u-\frac{1}{2}-d\big)$. Define $\CA^{(d)}=\CA_{\al_1,\al_2}(p^d_1,p^d_2)$.
Then there is a surjective algebra homomorphism
\[\varphi_d:U(\Fg)\to \CA^{(d)},\qquad 
e_i\mapsto X_i^+,\quad f_i\mapsto -X_i^-,\quad h_i\mapsto (-1)^i(H+d+1),\quad i=1,2.\]
This is related to irreducible $d$-dimensional evaluation representations of $\Fg$. See Example \ref{ex:11-example-area-d}.
\end{Example}

\begin{Example}[Finite W-algebra] \label{ex:finite-W}
\cite{DevVan1993,RagSor1998}
Let $\mathcal{W}=\mathcal{W}(\Fsl_4,\Fsl_2\oplus\Fsl_2)$ be the associative algebra with generators $w_2, J^a, S^a$ for $a\in\{+,-,0\}$ with $w_2$ central and
\begin{equation}
[J^a,J^b]=f^{ab}_{\phantom{ab}c}\, J^c\qquad 
[J^a,S^b]=f^{ab}_{\phantom{ab}c}\, S^c\qquad 
[S^a,S^b]=(w_2-c_2) f^{ab}_{\phantom{ab}c}\,J^c
\end{equation}
where $c_2=2(J^0J^0+J^+J^-+J^-J^+)$ is the quadratic Casimir. The structure constants are determined by $f^{+-}_{\phantom{+-}0}=f^{0+}_{\phantom{0+}+}=-f^{0-}_{\phantom{0-}-}=1$. Thus the $J^a$ span $\Fsl_2$ and the $S^a$ transform under the adjoint representation.
Let $\CA^{(d)}=\CA_{-1,1}\big((u-\frac{1}{2})(u-\frac{1}{2}-d),(u-\frac{1}{2})(u-\frac{1}{2}-d)\big)$ 
be the noncommutative Kleinian fiber product from Example \ref{ex:affine-A11}. Assume $d>1$. Then there is a homomorphism
\[\varphi:\mathcal{W}\to\CA^{(d)}\]
given by
\begin{subequations}
\begin{gather}
\varphi(J^\pm)=\frac{\pm 1}{\sqrt{2}}X_1^\pm 
\qquad \varphi(J^0)=-H+\frac{d+1}{2} \\ 
\varphi(S^-)=\frac{1}{\sqrt{2}}X_2^+ \qquad 
\varphi(S^0)=\frac{1}{2}[X_1^+,X_2^+] \qquad 
\varphi(S_+)=\frac{-1}{2\sqrt{2}}[X_1^+,[X_1^+,X_2^+]]\\
\varphi(w_2)=C^2+\frac{1}{2}(d^2-1)
\end{gather}
\end{subequations}
where
\begin{equation}\label{eq:w-ex-C}
C=\frac{1}{d^2-1}\big(X_2^+X_1^+(-2H+d-1)+X_1^+X_2^+(2H-d-3)\big)
\end{equation}
is the unique (up to sign), central element of $\CA^{(d)}$ of degree $(1,1)$ with normalization
\begin{equation}
C^\ast \cdot C= 1.
\end{equation}
Under the homomorphism $\varphi$, the Casimir $c_2$ is mapped to the scalar $\frac{1}{2}(d^2-1)$. 
Moreover, for any $\la\in \C^\times $, this induces a surjective homomorphism $\mathcal{W}\to  \CA^{(d)}/(C-\la)$.
For more details, see Appendix A.
\end{Example}

\begin{Example}[Lie-Heisenberg algebra] \label{ex:Lie-Heisenberg}
Let $\mathcal{H}=U(\Fsl_2\ltimes \Fh_3)$. Here $\Fsl_2=\C e\oplus \C h\oplus \C f$ acts naturally on the $3$-dimensional Heisenberg Lie algebra $\Fh_3$ by identifying $\Fh_3=V_2\oplus V_1$, where $V_2=\C x\oplus\C y$ and $V_1=\C z$ are the two- and one-dimensional $\Fsl_2$-irreps respectively, and $[x,y]=z$, $[z,x]=[z,y]=0$.
Thus $\mathcal{H}$ is the associative algebra with generators $e,f,h,x,y,z$ and defining relations:
\begin{subequations}
\begin{align}
[e,f]&=h & [h,e]&=2e & [h,f]&=-2f\\
[x,y]&=z & [x,z]&=0  & [y,z]&=0\\
[e,x]&=0 & [h,x]&=x  & [f,x]&=y\\
[e,y]&=x & [h,y]&=-y & [f,y]&=0\\
[e,z]&=0 & [h,z]&=0  & [f,z]&=0
\end{align}
\end{subequations}
Let $(\al_1,\al_2)=(-1,2)$ and $p_1(u)=u-\frac{1}{2}$, $p_2(u)=(u-1)u$.
Then there exists a homomorphism $\varphi:\mathcal{H}\to\CA_{\al_1,\al_2}(p_1,p_2)$ determined by:
\begin{subequations}
\begin{align}
\varphi(e)&=\frac{1}{2}X_2^+ &
\varphi(f)&=\frac{-1}{2}X_2^- &
\varphi(h)&=H-\frac{1}{2} \\
\varphi(x)&=\frac{1}{2}[X_2^+, X_1^+] &
\varphi(y)&=X_1^+ &
\varphi(z)&=\frac{1}{2}[[X_2^+,X_1^+],X_1^+]
\end{align}
\end{subequations}
\end{Example}

\subsection{Consistency}

In order to explain our main results and the connection with lattice models, we need to discuss the notion of consistency. For some choices of parameters $\al_1,\al_2,p_1,p_2$, the algebra $\CA_{\al_1,\al_2}(p_1,p_2)$ is the trivial algebra, an obviously undesirable property.
The following result resolves this problem.

\begin{Proposition} \label{prp:NKFP-Consistency}
Let $(\al_1,\al_2)\in\C^2$, $(p_1,p_2)\in(\C[u]\setminus\{0\})^2$, and put $\CA=\CA_{\al_1,\al_2}(p_1,p_2)$.
The following statements are equivalent.
\begin{enumerate}[{\rm (i)}]
\item $\CA\neq \{0\}$;
\item The generator $H$ is algebraically independent over $\C$ in $\tilde{\CA}$;
\item $(p_1,p_2)$ is a solution to the \emph{Mazorchuk-Turowska Equation (MTE)}
\begin{equation}\label{eq:MTeq}
p_1(u+\al_2/2)p_2(u+\al_1/2)=p_1(u-\al_2/2)p_2(u-\al_1/2).
\end{equation}
\end{enumerate}
\end{Proposition}
\begin{proof}
(ii)$\Rightarrow$(i): Obvious.

(i)$\Rightarrow$(ii): If $f(H)=0$ in $\tilde{\CA}$ for some nonzero polynomial $f$ then $1\in\CI$, hence $\CA=\{0\}$.

(ii)$\Leftrightarrow$(iii): Since $\CA$ is an example of a TGWA (see Corollary \ref{cor:fiber-TGWA}), this follows directly from the main theorem of \cite{FutHar2012a}. However, for the convenience of the reader we give a proof of  (ii)$\Rightarrow$(iii). This direction was shown (more generally) by Mazorchuk and Turowska in \cite{MazTur1999}. In $\tilde{\CA}$ we have, using the relations \eqref{eq:Aalbepq-rels},
\[X_1^+(X_2^+X_1^-)X_1^+ = X_1^+(X_1^-X_2^+)X_1^+,\]
\begin{equation}\label{eq:consistency-pf-1a}
X_1^+X_2^+ p_1(H+\al_1/2) = p_1(H-\al_1/2)X_2^+X_1^+.
\end{equation}
Symmetrically, by interchanging the subscripts $1$ and $2$, we have
\begin{equation}\label{eq:consistency-pf-1b}
X_2^+X_1^+ p_2(H+\al_2/2) = p_2(H-\al_2/2)X_1^+X_2^+.
\end{equation}
Combining \eqref{eq:consistency-pf-1a}-\eqref{eq:consistency-pf-1b} we get:
\[p_1(H-\al_1/2)p_2(H-\al_2/2) X_1^+X_2^+ = X_1^+X_2^+ p_1(H+\al_1/2)p_2(H+\al_2/2).\]
Put $h=H-(\al_1+\al_2)/2$ and use $X_i^+H=(H-\al_i)X_i^+$:
\[ \big(p_1(h+\al_2/2)p_2(h+\al_1/2) - p_1(h-\al_2/2)p_2(h-\al_1/2)\big)\cdot X_1^+X_2^+ = 0.\]
Multiplying from the right by $X_1^-X_2^-$ we obtain
\[ \big(p_1(h+\al_2/2)p_2(h+\al_1/2) - p_1(h-\al_2/2)p_2(h-\al_1/2)\big)\cdot p_1(H-\al_1/2)p_2(H-\al_2/2) = 0.\]
Thus, if $H$ is algebraically independent over $\C$ in $\tilde{\CA}$, then \eqref{eq:MTeq} holds identically in the polynomial ring $\C[u]$.
\end{proof}

\begin{Assumption}
In the rest of this paper, whenever we consider a noncommutative Kleinian fiber product $\CA_{\al_1,\al_2}(p_1,p_2)$, we implicitly assume that \eqref{eq:MTeq} holds.
\end{Assumption}

\subsection{Classification of solutions}

Before embarking on an investigation of the algebras $\CA_{\al_1,\al_2}(p_1,p_2)$ it is important to understand the solutions $(p_1,p_2)$ to the MTE \eqref{eq:MTeq}. This problem  has a very satisfying answer, and this is where the connection to lattice models comes in. In \cite{HarRos2016}, the author and D. Rosso obtained a complete classification of all solutions in terms of generalized Dyck paths. It is well-known that such paths are in bijection with six-vertex or higher spin vertex configurations appearing in statistical mechanics \cite{Bax2007,GomRuiSie2005,Zin2009}.
In this subsection we slightly reformulate the classification in terms of those. 

Consider the two-dimensional face-centered unit square lattice (Figure \ref{fig:lattice}) and put:
\begin{alignat*}{2}
\mathsf{F}&=\Z^2 &&\qquad\text{midpoints of faces, marked \tikz[baseline=-.5ex]{\draw (-2pt,-2pt)--(2pt,2pt);\draw(2pt,-2pt)--(-2pt,2pt);}}\\
\mathsf{V}&=\mathsf{F}+(1/2,1/2) &&\qquad\text{vertices, marked \tikz[baseline=-.5ex]{\fill (0,0) circle (2pt);}}\\
\mathsf{E_1}&=\mathsf{F}+(1/2,0) &&\qquad\text{midpoints of vertical edges}\\
\mathsf{E_2}&=\mathsf{F}+(0,1/2) &&\qquad\text{midpoints of horizontal edges}\\
\mathsf{E} &=\mathsf{E_1}\cup \mathsf{E_2} &&
\end{alignat*}

\begin{figure}
\centering 
\begin{tikzpicture}
% Help lines
\foreach \x in {-2.5,-1.5,...,2.5}{
 \draw[help lines] (\x,-3) to (\x,3);
 \draw[help lines] (-3,\x) to (3,\x); }
% Vertices
\foreach \x in {-2.5,-1.5,...,2.5}{
 \foreach \y in {-2.5,-1.5,...,2.5}{
  \fill (\x,\y) circle (2pt); } }
% Faces
\foreach \x in {-2,...,2}{
 \foreach \y in {-2,...,2}{
  \draw (\x cm-2pt,\y cm -2pt)--(\x cm+2pt,\y cm + 2pt);
  \draw (\x cm+2pt,\y cm -2pt)--(\x cm-2pt,\y cm + 2pt); } }
% Face label
\node[font=\scriptsize, below] at (0,0) {$(0,0)$};
\end{tikzpicture}
\caption{The two-dimensional face-centered unit square lattice.}
\label{fig:lattice}
\end{figure}

\begin{Definition} Let $(m,n)\in\Z^2$. An \emph{$(m,n)$-periodic higher spin vertex configuration} $\mathscr{L}$ is a function $\mathscr{L}: \mathsf{E} \to \N$ assigning a non-negative integer label $\mathscr{L}(e)$ to each edge $e\in \mathsf{E}$, such that
\begin{enumerate}[{\rm (i)}]
\item (\emph{periodicity}) for every $e\in \mathsf{E}$ we have $\mathscr{L}\big(e+(m,n)\big)=\mathscr{L}(e)$;
\item (\emph{finiteness}) 
for every $e\in \mathsf{E}$ we have $\mathscr{L}\big(e+k(-n,m)\big)=0$ for $|k|\gg 0$;
\item (\emph{local current conservation}) for every $v\in \mathsf{V}$ we have (see Figure \ref{fig:vertex}):
\begin{equation}\label{eq:current-conservation}
 \mathscr{L}\big(v-(1/2,0)\big)+\mathscr{L}\big(v-(0,1/2)\big) = \mathscr{L}\big(v+(1/2,0)\big)+\mathscr{L}\big(v+(0,1/2)\big).
\end{equation}
\end{enumerate}
We call $\mathscr{L}$ a \emph{six-vertex configuration} if $\mathscr{L}(e)\in\{0,1\}$ for all $e\in \mathsf{E}$, and  \emph{trivial} if $\mathscr{L}(e)=0$ for all $e\in \mathsf{E}$.
\end{Definition}

\begin{figure}
\centering 
\[
\begin{tikzpicture}[baseline={([yshift=-.5ex]current bounding box.center)}]
\draw (-1.6,0)--(1.6,0);
\draw[->] (-1.6,0)--(-.8,0);
\draw[->] (0,0)--(.8,0);
\draw ( 0,-1.6)--( 0, 1.6);
\draw[->] (0,-1.6)--(0,-.8);
\draw[->] (0,0)--(0,.8);
\fill (0,0) circle (2pt);
\node[font=\scriptsize, anchor=north west] at (0,0) {$v$};
\node[font=\scriptsize, anchor=south] at (-.8,0) {$a$};
\node[font=\scriptsize, anchor=south] at ( .8,0) {$b$};
\node[font=\scriptsize, anchor=east] at (0, .8) {$c$};
\node[font=\scriptsize, anchor=east] at (0,-.8) {$d$};
\end{tikzpicture}
\qquad a+d=b+c
\]
\caption{Local current conservation (ice rule).}\label{fig:vertex}
\end{figure}
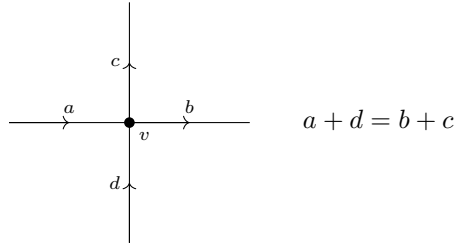

Such a configuration $\mathscr{L}$ can be viewed as a multiset of non-crossing vertex paths of period $(m,n)$, see Figure \ref{fig:53-example-a}. It is natural to think of $\mathscr{L}(e)$ as the ``multiplicity'' of $e$. For example, $\mathscr{L}(e)=0$ then means $e$ is absent from the configuration, while $\mathscr{L}(e)=2$ means $e$ is a double edge.

To each configuration $\mathscr{L}$ one can attach a solution $(P^\mathscr{L}_1,P^\mathscr{L}_2)$ to the MTE \eqref{eq:MTeq} as follows.
Let $\Ga_{m,n}=\langle(m,n)\rangle$ be the infinite cyclic subgroup of the additive group $\R^2$ with generator $(m,n)$. The group $\Ga_{m,n}$ acts by translations on the sets $\mathsf{E}_i$.
For $(\al_1,\al_2)\in\C^2$ with $m\al_1+n\al_2=0$ and an $(m,n)$-periodic higher spin vertex configuration $\mathscr{L}$, put for $i=1,2$:
\begin{equation}\label{eq:fundamental-solution}
P^{\mathscr{L}}_i(u)=P^{\mathscr{L}}_i(u;\al_1,\al_2)=\prod_{(x_1,x_2)+\Ga_{m,n}\in \mathsf{E}_i/\Ga_{m,n}} \big(u-(x_1\al_1+x_2\al_2)\big)^{\mathscr{L}(x_1,x_2)}.
\end{equation}
By periodicity of $\mathscr{L}$ and that $m\al_1+n\al_2=0$ the expression is independent of the choice of representatives $(x_1,x_2)$ modulo $\Ga_{m,n}$. By finiteness, $P^\mathscr{L}_i(u)$ are polynomials in $u$. Using the local current conservation \eqref{eq:current-conservation} one verifies that $(P^\mathscr{L}_1,P^\mathscr{L}_2)$ solves the MTE \eqref{eq:MTeq}.

\begin{Example}
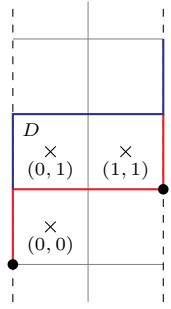
\begin{figure}
\centering
\begin{tikzpicture}[xscale=1,yscale=1]
% Help lines
\draw[help lines] (1,-.5)--(1,3.5);
\draw[help lines] (0,0)--(2,0);
\draw[help lines] (0,1)--(2,1);
\draw[help lines] (0,2)--(2,2);
\draw[help lines] (0,3)--(2,3);
% Boundary lines
\draw[dashed] (0,-.5)--(0,3.5);
\draw[dashed] (2,-.5)--(2,3.5);
% First path
\draw[-,thick,Red] (0,0)--(0,1)--(1,1)--(2,1)--(2,2);
% Second path
\draw[-,thick,Blue] (0,1)--(0,2)--(1,2)--(2,2)--(2,3);
% Vertices
\fill (0,0) circle (2pt);
\fill (2,1) circle (2pt);
% Weight spaces
\draw (.5cm-2pt,.5cm-2pt)--(.5cm+2pt,.5cm+2pt);
\draw (.5cm+2pt,.5cm-2pt)--(.5cm-2pt,.5cm+2pt);
\draw (.5cm-2pt,1.5cm-2pt)--(.5cm+2pt,1.5cm+2pt);
\draw (.5cm+2pt,1.5cm-2pt)--(.5cm-2pt,1.5cm+2pt);
\draw (1.5cm-2pt,1.5cm-2pt)--(1.5cm+2pt,1.5cm+2pt);
\draw (1.5cm+2pt,1.5cm-2pt)--(1.5cm-2pt,1.5cm+2pt);
\node[font=\scriptsize, below] at (.5,.5) {$(0,0)$};
\node[font=\scriptsize, below] at (.5,1.5) {$(0,1)$};
\node[font=\scriptsize, below] at (1.5,1.5) {$(1,1)$};
\node[font=\scriptsize, below right] at (0,2) {$D$};
\end{tikzpicture}
\caption{A $(2,1)$-periodic six-vertex configuration drawn in the strip $[\frac{-1}{2},\frac{-1}{2}+2]\times\R$, the closure of a fundamental domain with respect to translations by the vector $(2,1)$.}
\label{fig:21-config}
\end{figure}
Let $(m,n)=(2,1)$ and $\mathscr{L}$ be the $(2,1)$-periodic six-vertex configuration in Figure \ref{fig:21-config}.
Put $(\al_1,\al_2)=(-1,2)$. Up to translations by $(2,1)$, there are exactly two vertical edges in $\mathscr{L}$, each of multiplicity one, with midpoints $\textcolor{Red}{(-\frac{1}{2},0)}$ and $\textcolor{Blue}{(-\frac{1}{2},1)}$ respectively. Thus
\begin{align*}
P^\mathscr{L}_1(u)&=\textcolor{Red}{\big(u-(-\tfrac{1}{2}\al_1)\big)}\textcolor{Blue}{\big(u-(-\tfrac{1}{2}\al_1+\al_2)\big)} \\ 
&= (u-\tfrac{1}{2})(u-\tfrac{5}{2}).
\end{align*}
Similarly, there are four horizontal edges in $\mathscr{L}$, all of multiplicity one, with midpoints $\textcolor{Red}{(0,\frac{1}{2}), (1,\frac{1}{2})}$, $\textcolor{Blue}{(0,\frac{3}{2}), (1,\frac{3}{2})}$ respectively. Thus
\begin{align*}
P^\mathscr{L}_2(u) &=\textcolor{Red}{\big(u-\tfrac{1}{2}\al_2\big)\big(u-(\al_1+\tfrac{1}{2}\al_2)\big)}\textcolor{Blue}{\big(u-\tfrac{3}{2}\al_2\big)\big(u-(\al_1+\tfrac{3}{2}\al_2)\big)} \\
&=(u-1)u(u-3)(u-2).
\end{align*}
One can directly verify that \eqref{eq:MTeq} holds.
\end{Example}

Conversely, any solution is uniquely a product of shifts of those, in the following way.

\begin{Theorem}[Reformulation of the main result of \cite{HarRos2016}]\label{thm:HarRos-classification}
Fix $(\al_1,\al_2)\in\C^2\setminus\{(0,0)\}$.
\begin{enumerate}[{\rm (a)}]
\item Let $(p_1,p_2)\in(\C[u]\setminus\{0\})^2$ be any solution to the MTE \eqref{eq:MTeq} where $p_1$ and $p_2$ are monic, and not both constant. Then there exist a unique pair $(m,n)$ of relatively prime non-negative integers with $m\al_1+n\al_2=0$,
a unique non-negative integer $k$, complex numbers $\la_1,\la_2,\ldots,\la_k$ pairwise incongruent modulo $\Z\al_1+\Z\al_2$, and non-trivial $(m,n)$-periodic higher spin vertex configurations $\mathscr{L}^{(1)},\mathscr{L}^{(2)},\ldots,\mathscr{L}^{(k)}$ such that for $i=1,2$:
\begin{equation} \label{eq:HR-factorization}
p_i(u) = P^{\mathscr{L}^{(1)}}_i(u-\la_1) P^{\mathscr{L}^{(2)}}_i(u-\la_2)\cdots P^{\mathscr{L}^{(k)}}_i(u-\la_k).
\end{equation}
\item The set $\big\{[\mathscr{L}^{(i)},\la_i]\mid i=1,2,\ldots,k\big\}$ is uniquely determined by $(p_1,p_2)$
where $[\mathscr{L},\la]$ is the orbit containing $(\mathscr{L},\la)$ under the action of the group $\Z^2$ given by
 $\mu \cdot (\mathscr{L},\la)=(\mathscr{L}^\mu,\la + \mu_1\al_1+\mu_2\al_2)$, $\mathscr{L}^\mu(e)=\mathscr{L}(e+\mu)$ for $e\in \mathsf{E}$, $\mu=(\mu_1,\mu_2)\in\Z^2$.
\end{enumerate}
\end{Theorem}

\section{Main results}\label{sec:main-results}

In this section we state our three main theorems.

\subsection{Reduction to integral weight modules over $\CA(\mathscr{L})$}
\label{sec:thmA}
The first theorem essentially shows that the identity \eqref{eq:HR-factorization} can be lifted to the level of weight module categories. 

Let $(m,n)$ be a pair of relatively prime non-negative integers, and $\mathscr{L}$ be an $(m,n)$-periodic higher spin vertex configuration. Define 
\begin{equation}\label{eq:AL-def}
\CA(\mathscr{L})=\CA_{-n,m}\big(P^\mathscr{L}_1(u;-n,m), P^\mathscr{L}_2(u;-n,m)\big),
\end{equation}
where $P^\mathscr{L}_i(u)$ were defined in \eqref{eq:fundamental-solution}.

\begin{Definition}
A module $M$ over a noncommutative Kleinian fiber product $\CA_{\al_1,\al_2}(p_1,p_2)$ is called a \emph{weight module} if 
\[M=\bigoplus_{\la\in\C} M_\la,\qquad M_\la=\{v\in M\mid Hv=\la v\}.\]
The \emph{support} of $M$ is $\Supp(M)=\{\la\in\C\mid M_\la\neq 0\}$. 
Let $\mathscr{W}_{\al_1,\al_2}(p_1,p_2)$ denote the category of weight modules over $\CA_{\al_1,\al_2}(p_1,p_2)$.
A weight module $M$ over $\CA(\mathscr{L})$ is called \emph{integral} if $\Supp(M)\subseteq \Z$.
Let $\mathscr{W}(\mathscr{L})_\Z$ denote the category of integral weight modules over $\CA(\mathscr{L})$.
\end{Definition}

The first main result reduces the problem of classifying simple weight modules over $\CA_{\al_1,\al_2}(p_1,p_2)$ to that of classifying simple integral weight modules over $\CA(\mathscr{L})$.

\begin{MainTheorem}\label{thm:A}
Let $\CA_{\al_1,\al_2}(p_1,p_2)$ be any non-trivial noncommutative Kleinian fiber product where $(\al_1,\al_2)\in\C^2\setminus\{(0,0)\}$. Then there exist a pair $(m,n)$ of relatively prime non-negative integers with $m\al_1+n\al_2=0$ and a sequence of $(m,n)$-periodic higher spin vertex configurations $\mathscr{L}^{(\om)}$ indexed by $\om\in\C/\Z$, at most finitely many non-trivial, such that there is an equivalence of categories
\begin{equation}
\mathscr{W}_{\al_1,\al_2}(p_1,p_2)\simeq \prod_{\om\in\C/\Z} \mathscr{W}(\mathscr{L}^{(\om)})_\Z.
\end{equation}
\end{MainTheorem}

The proof will be given in Section \ref{sec:ProofA}.

\subsection{Classification of simple integral weight $\CA(\mathscr{L})$-modules}
\label{sec:thmB}

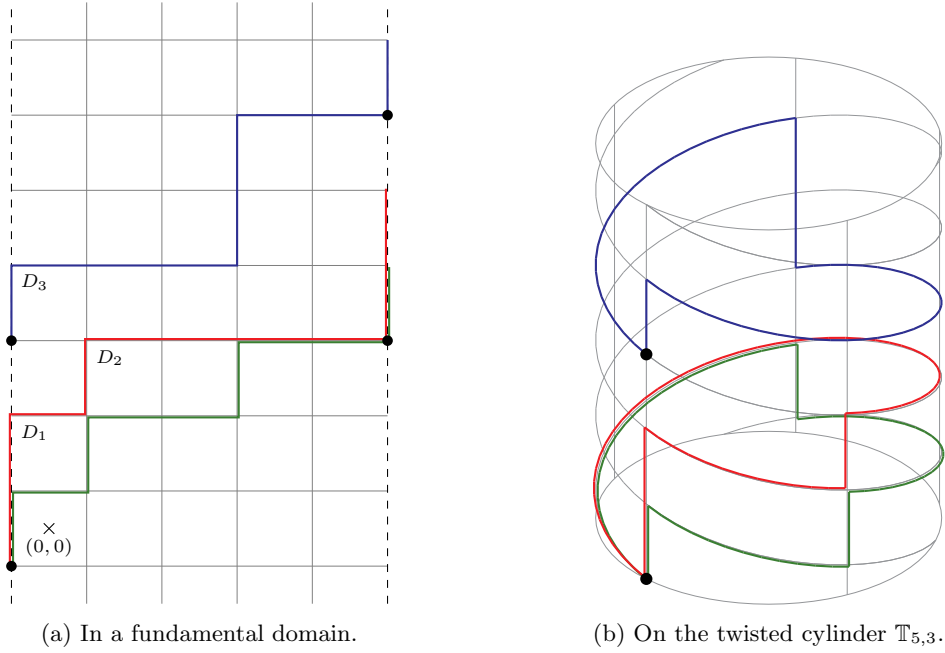
\begin{figure}
\centering
\begin{subfigure}[b]{.4\textwidth}
\centering 
\begin{tikzpicture}
% Grid
\foreach \y in {0,1,...,7} { \draw[help lines] (0,\y)--(5,\y); }
\foreach \x in {1,...,4} { \draw[help lines] (\x,-.5)--(\x,7.5); }
% Seams
\draw[dashed] (0,-.5) to (0,7.5);
\draw[dashed] (5,-.5) to (5,7.5);
% Vertex paths
\draw[thick,OliveGreen,xshift=.6pt,yshift=-.6pt] (0,0)--(0,1)--(1,1)--(1,2)--(2,2)--(3,2)--(3,3)--(4,3)--(5,3)--(5,4);
\draw[thick,Red,xshift=-.6pt,yshift=.6pt] (0,0)--(0,1)--(0,2)--(1,2)--(1,3)--(2,3)--(3,3)--(4,3)--(5,3)--(5,5);
\draw[thick,Blue] (0,3)--(0,4)--(1,4)--(2,4)--(3,4)--(3,5)--(3,6)--(4,6)--(5,6)--(5,7);
% Some marked vertices
\fill (0,0) circle (2pt);
\fill (5,3) circle (2pt);
\fill (0,3) circle (2pt);
\fill (5,6) circle (2pt);
% A marked face point
\draw (.5cm-2pt,.5cm-2pt)--(.5cm+2pt,.5cm+2pt);
\draw (.5cm+2pt,.5cm-2pt)--(.5cm-2pt,.5cm+2pt);
\node[font=\scriptsize, below] at (.5,.5) {$(0,0)$};
% Notation for connected components
\node[font=\scriptsize, below right] at (0,2) {$D_1$}; 
\node[font=\scriptsize, below right] at (1,3) {$D_2$}; 
\node[font=\scriptsize, below right] at (0,4) {$D_3$}; 
\end{tikzpicture}
\caption{In a fundamental domain.}
\label{fig:53-example-a}
\end{subfigure}
\qquad\qquad
\begin{subfigure}[b]{.4\textwidth}
\centering
\begin{tikzpicture}[x={(-.7071067812cm,-.35355339059327373cm)},
                    y={(.7071067812cm,-.35355339059327373cm)},
                    z={(0cm, .8660254037844387cm)},
                    xscale=1.15, yscale=1.15]
\setupcylinder{5}{3}{5} % {m}{n}{Vertical grid height} where (m,n) is the period of the cylinder
\drawlatticepath{0}{0}{1,0,1,0,0,1,0,0}{OliveGreen,thick,,xshift=.6pt,yshift=-.6pt}
\drawlatticepath{0}{0}{1,1,0,1,0,0,0,0}{Red,thick,xshift=-.6pt,yshift=.6pt}
\drawlatticepath{0}{3}{1,0,0,0,1,1,0,0}{Blue,thick}
\drawlatticedot{0}{0}
\drawlatticedot{0}{3}
\end{tikzpicture}
\caption{On the twisted cylinder $\T_{5,3}$.}
\label{fig:53-example-b}
\end{subfigure}
\caption{A $(5,3)$-periodic higher spin vertex configuration $\mathscr{L}$.}
\label{fig:53-example}
\end{figure}

Let $(m,n)$ be a pair of relatively prime non-negative integers and $\mathscr{L}$ be an $(m,n)$-periodic higher spin vertex configuration.
For $i\in\{1,2\}$ and $e=(x_1,x_2)\in \mathsf{E}_i$, recalling that $e$ is by definition the midpoint of an edge, let $[e]$ be the corresponding closed line segment in $\R^2$:
\[ [e] = \big\{(a_1,a_2)\in\R^2\mid a_i=x_i\text{ and } x_{3-i}-\tfrac{1}{2}\le a_{3-i}\le x_{3-i}+\tfrac{1}{2} \big\}.\]
Let $[\mathscr{L}]\subseteq \R^2$ denote the union of the edges that appear in $\mathscr{L}$:
\[ [\mathscr{L}]=\bigcup_{\substack{e\in \mathsf{E}\\ \mathscr{L}(e)>0}} [e]. \]
Consider the twisted cylinder $\T_{m,n}=\R^2/\Ga_{m,n}$ where $\Ga_{m,n}=\langle(m,n)\rangle$ with quotient topology making $\T_{m,n}$ homeomorphic to $S_1\times\R$.
Let $\overline{\mathscr{L}}\subseteq \T_{m,n}$ be the image of $[\mathscr{L}]$ under the canonical projection $\R^2\to\T_{m,n}$ (see Figure \ref{fig:53-example-b}). Removing the set $\overline{\mathscr{L}}$ from $\T_{m,n}$ yields a ``crackle cylinder'' $\T_{m,n}\setminus\overline{\mathscr{L}}$ in which every connected component is either contractible or homotopic to a circle $S_1\times \ast$. 
Finally, let $\overline{\mathsf{F}}$ be the image of the set of face midpoints $\mathsf{F}=\Z^2$ under the canonical projection $\R^2\to\T_{m,n}$.

\begin{MainTheorem}\label{thm:B}
Let $(m,n)$ be a pair of relatively prime non-negative integers, $\mathscr{L}$ be an $(m,n)$-periodic higher spin vertex configuration, and $\CA=\CA(\mathscr{L})$ be the corresponding noncommutative Kleinian fiber product \eqref{eq:AL-def}. 
\begin{enumerate}[{\rm (i)}]
\item There is a bijective correspondence between the set of  isoclasses of simple integral weight $\CA$-modules, and the set of pairs $(D,\xi)$ where $D$ is a connected component of $\T_{m,n}\setminus\overline{\mathscr{L}}$ and $\xi\in\C$ with $\xi=0$ iff $D$ is contractible.
\item Let $M(D,\xi)$ be the module corresponding to $(D,\xi)$. Each nonzero weight space $M(D,\xi)_\la$ is one-dimensional and 
\[\Supp\big(M(D,\xi)\big)=\{x_1\al_1+x_2\al_2\mid (x_1,x_2)+\Ga_{m,n}\in \overline{\mathsf{F}}\cap D\}.\]
In particular, $\dim M(D,\xi)=\mathrm{area}(D)$.
\item $\xi$ can be chosen such that
$C|_{M(D,\xi)}=\xi\Id_{M(D,\xi)}$ where $C\in\CA\otimes_{\C[H]}\C(H)$ is a certain $\CA$-centralizing element \eqref{eq:C} acting on any $M(D,\xi)$ with incontractible $D$.
\item Any finite-dimensional simple $\CA$-module is an integral weight module.
\end{enumerate}
\end{MainTheorem}

The key technical result in order to establish Theorem \ref{thm:B}(i) is that in any subquotient $C_{\CA}(H)/(H-\la)$, the gradation radical is equal to the nil radical. 
In the language of TGWAs \cite{MazTur1999,Har2006}, we prove that if $M$ is a simple weight $\CA$-module, then $M$ has no inner breaks. The we can apply the general classification theorem for such modules, see Theorem \ref{thm:TGWA-modules-classification}.

\begin{Example} Let $(m,n)=(2,1)$ and $\mathscr{L}$ be given by Figure \ref{fig:21-config}. Then $\CA(\mathscr{L})$ has a unique finite-dimensional simple module, namely $M(D,0)$, and $\dim M(D,0)=2$. $\overline{\mathsf{F}}\cap D$ contains two points $(0,1)+\Ga_{2,1}$ and $(1,1)+\Ga_{2,1}$. Since $0\al_1+1\al_2=2$ and $1\al_1+1\al_2=1$, $\Supp\big(M(D,0)\big)=\{2,1\}$, thus the module $M(D,0)$ has a basis in which $H$ is represented by the diagonal matrix $\left[\begin{smallmatrix}2&0\\0&1\end{smallmatrix}\right]$.
\end{Example}

\begin{Example}
Let $\mathscr{L}$ be as in Figure \ref{fig:53-example}. Then $P^\mathscr{L}_1(u)$ has degree $9$, while $P^\mathscr{L}_2(u)$ has degree $15$. Cutting the cylinder along the edges yields five connected components, two of which are infinite (top and bottom). The remaining three correspond to all finite-dimensional simple $\CA(\mathscr{L})$-modules: a unique one-dimensional module $M(D_1,0)$, a unique two-dimensional module $M(D_2,0)$, and a one-parameter family of pairwise non-isomorphic $10$-dimensional modules $M(D_3,\xi)$, $\xi\in\C^\times$.
\end{Example}

\begin{Example} \label{ex:11-example-area-d}
Consider the algebra $\CA(\mathscr{L})$ with $\mathscr{L}$ as in Figure \ref{fig:11-d}. With $(\al_1,\al_2)=(-1,1)$ we have $P^\mathscr{L}_1(u)=P^\mathscr{L}_2(u)=\big(u-(-\frac{1}{2})\big)\big(u-(-\frac{1}{2}-d)\big)$. Thus this is the same as the algebra $\CA^{(d)}$ from Examples \ref{ex:affine-A11} and \ref{ex:finite-W}. Then every finite-dimensional simple $\CA(\mathscr{L})$-module has dimension $d$ and there is a one-parameter family of such modules.
\begin{figure}
\centering
\[
\begin{tikzpicture}
\foreach \y in {0,...,5}{
\draw[help lines] (-.5,\y)--(1.5,\y);
}
\draw[dashed] (0,-.5)--(0,5.5);
\draw[dashed] (1,-.5)--(1,5.5);
\draw[thick,Blue] (-.5,0)--(0,0)--(0,1)--(1,1)--(1,2)--(1.5,2);
\draw[thick,Blue] (-.5,3)--(0,3)--(0,4)--(1,4)--(1,5)--(1.5,5);
\fill (0,0) circle (2pt);
\fill (1,1) circle (2pt);
\fill (0,3) circle (2pt);
\fill (1,4) circle (2pt);
\node[font=\scriptsize, anchor=south] at (.5,.5) {$0$}; 
\node[font=\scriptsize, anchor=south] at (.5,1.5) {$1$}; 
\node[font=\scriptsize, anchor=south] at (.5,3.5) {$d$}; 
\node[font=\scriptsize] at (.5,2.5) {$\vdots$}; 
\draw (.5cm-2pt,.5cm-2pt)--(.5cm+2pt,.5cm+2pt);
\draw (.5cm+2pt,.5cm-2pt)--(.5cm-2pt,.5cm+2pt);
\draw (.5cm-2pt,1.5cm-2pt)--(.5cm+2pt,1.5cm+2pt);
\draw (.5cm+2pt,1.5cm-2pt)--(.5cm-2pt,1.5cm+2pt);
\draw (.5cm-2pt,3.5cm-2pt)--(.5cm+2pt,3.5cm+2pt);
\draw (.5cm+2pt,3.5cm-2pt)--(.5cm-2pt,3.5cm+2pt);
\end{tikzpicture}
\]
\caption{A $(1,1)$-periodic six-vertex configuration $\mathscr{L}$. The face points $(x_1,x_2)$ are labeled by the corresponding $H$-eigenvalue $x_1\al_1+x_2\al_2=x_2-x_1$. The corresponding noncommutative Kleinian fiber product $\CA(\mathscr{L})$ is related to the affine Lie algebra $A_1^{(1)}$ and the finite W-algebra $\mathcal{W}(\Fsl_4,\Fsl_2\oplus\Fsl_2)$, See Examples \ref{ex:affine-A11} and \ref{ex:finite-W}.}
\label{fig:11-d}
\end{figure}
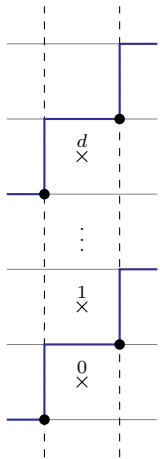
\end{Example}

\subsection{Description of the center of $\CA(\mathscr{L})$}
Our third and final main result is the description of the center of $\CA(\mathscr{L})$. To state it we need the notion of a five-vertex configuration.

\begin{Definition}\label{def:five-vertex}
A higher spin vertex configuration $\mathscr{L}$ is called a \emph{five-vertex configuration} if for every vertex $v\in V$, at most two of the four incident edges have non-zero label. Thus at every vertex $v$ there are only five types of allowed local configurations (Figure \ref{fig:five-vertex}). 
\end{Definition}

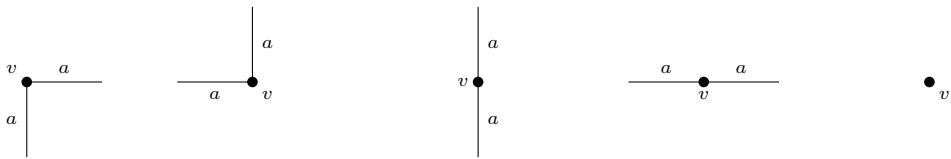
\begin{figure}[h]
\centering 
\begin{tikzpicture}
\draw (0,-1)--(0,0)--(1,0);
\fill (0,0) circle (2pt);
\node[font=\scriptsize, anchor=south east] at (0,0) {$v$};
\node[font=\scriptsize, anchor=south] at (.5,0) {$a$};
\node[font=\scriptsize, anchor=east] at (0,-.5) {$a$};
\draw (2,0)--(3,0)--(3,1);
\fill (3,0) circle (2pt);
\node[font=\scriptsize, anchor=north west] at (3,0) {$v$};
\node[font=\scriptsize, anchor=north] at (2.5,0) {$a$};
\node[font=\scriptsize, anchor=west] at (3,.5) {$a$};
\draw (6,-1)--(6,1);
\fill (6,0) circle (2pt);
\node[font=\scriptsize, anchor=east] at (6,0) {$v$};
\node[font=\scriptsize, anchor=west] at (6,.5) {$a$};
\node[font=\scriptsize, anchor=west] at (6,-.5) {$a$};
\draw (8,0)--(10, 0);
\fill (9,0) circle (2pt);
\node[font=\scriptsize, anchor=north] at (9,0) {$v$};
\node[font=\scriptsize, anchor=south] at (8.5,0) {$a$};
\node[font=\scriptsize, anchor=south] at (9.5,0) {$a$};
\fill (12,0) circle (2pt);
\node[font=\scriptsize, anchor=north west] at (12,0) {$v$};
\end{tikzpicture}
\caption{Local five-vertex configurations. The edge label $a$ can be any positive integer.}
\label{fig:five-vertex}
\end{figure}

For an example, see Figure \ref{fig:43-five-vertex}.

\begin{MainTheorem}\label{thm:C}
Let $(m,n)$ be a pair of relatively prime non-negative integers, $\mathscr{L}$ be an $(m,n)$-periodic higher spin vertex configuration, and $\CA(\mathscr{L})$ be the corresponding noncommutative Kleinian fiber product \eqref{eq:AL-def}.
\begin{enumerate}[{\rm (a)}]
\item If $\mathscr{L}$ is a five-vertex configuration, then the center of $\CA(\mathscr{L})$ is a Laurent polynomial algebra in one variable:
\[
Z\big(\CA(\mathscr{L})\big)=\C[C,C^{-1}],
\]
where $C$ is an element of degree $(m,n)$, explicitly given in \eqref{eq:C}.
\item If $\mathscr{L}$ is not a five-vertex configuration, then $Z(\CA(\mathscr{L}))=\C$.
\end{enumerate}
\end{MainTheorem}

\section{Twisted generalized Weyl algebras}\label{sec:tgwa}
As we prove below, noncommutative Kleinian fiber products are examples of twisted generalized Weyl algebras, introduced by Mazorchuk and Turowska \cite{MazTur1999}. In this section we recall their definition, prove several new results and clarify some details about the classification of simple weight modules without inner breaks from \cite{Har2006}. These results will be applied to the proofs of the main theorems in the sections that follow.

Throughout this section we work over an arbitrary ground field $\K$.

\subsection{Definitions}
Below we have modified the traditional definitions slightly by using square roots $\si_i$.
This allows for a more symmetric presentation and conceptual notion of breaks for weight modules.
However, it is merely a notational device because any identity such as \eqref{eq:TGWA-Consistency-Rels} can be rewritten without square roots by substituting $t_i'=\si_i^{-1/2}(t_i)$.

\begin{Definition}[TGWC \cite{MazTur1999}] \label{def:TGWC}
Let $n$ be a positive integer, $R$ an associative unital $\K$-algebra, $\Bsi=(\sigma_1^{1/2},\sigma_2^{1/2},\ldots,\sigma_n^{1/2})\in \Aut_\K(R)^n$ an $n$-tuple of commuting automorphisms of $R$, $\Bt=(t_1,t_2,\ldots,t_n)\in Z(R)^n$ an $n$-tuple of central elements in $R$. The corresponding \emph{twisted generalized Weyl construction (TGWC) of rank $n$}, denoted $\tilde{\CA}(R,\Bsi,\Bt)$, is the associative algebra obtained from $R$ by adjoining $2n$ new generators
$X_1^\pm,X_2^\pm,\ldots,X_n^\pm$ that are not required to commute with each other, nor with the elements of $R$, but are subject to the following relations for all $i,j=1,2,\ldots,n$:
\begin{equation}\label{eq:TGWA-Rels}
X_i^\pm  r = \si_i^{\pm 1}(r) X_i^\pm,\qquad 
X_i^\pm X_i^\mp = \si_i^{\pm 1/2}(t_i), \qquad 
[X_i^\pm, X_j^\mp ] =0 \quad \text{if $i\neq j$}.
\end{equation}  
\end{Definition}
Note that that we do not impose $[X_i^\pm,X_j^\pm]=0$. Instead we mod out by an ideal. Relations \eqref{eq:TGWA-Rels} imply that $\tilde{\CA}=\tilde{\CA}(R,\Bsi,\Bt)$ has a $\Z^n$-gradation $\tilde{\CA}=\bigoplus_{g\in \Z^n} \tilde{\CA}_g$ induced by $\deg r=0$ for $r\in R$ and $\deg X_i^{\pm}=\pm \Be_i$, where $\{\Be_i\}_{i=1}^n$ is the standard $\Z$-basis for $\Z^n$. One checks that $\tilde{\CA}_0$ coincides with the image of $R$ under the canonical homomorphism $R\to\tilde{\CA}$, $r\mapsto r1_{\tilde{\CA}}$.

\begin{Definition}[TGWA \cite{MazTur1999}] \label{def:TGWA}
Let $n,R,\Bsi,\Bt$ be as in Definition \ref{def:TGWC}. The corresponding \emph{twisted generalized Weyl algebra (TGWA)}, denoted $\CA(R,\Bsi,\Bt)$, is defined as the quotient $\tilde{\CA}/\CI$, where $\CI$ is the sum of all graded ideals $I=\bigoplus_{g\in\Z^n} I_g$ in $\tilde{\CA}$ with $I_0=0$.
\end{Definition}
\begin{Remark}
The last relation in \eqref{eq:TGWA-Rels} can be generalized to
\begin{equation}
 q_{ij}^\pm X_i^\pm X_j^\mp = q_{ji}^\mp X_j^\mp X_i^\pm,\qquad i\neq j,
\end{equation}
where $q_{ij}^\pm\in\K^\times$ (see \cite{MazPonTur2003}). All results in this section carry over with only minor modifications to such ``quantum'' TGWAs.
\end{Remark}

\begin{Remark} Using only Relations \eqref{eq:TGWA-Rels}, one can deduce the following \emph{exchange relation}:
\begin{equation}\label{eq:TGWA-exchange-relation}
X_i^\pm X_j^\pm \si_i^{\mp 1/2}(t_i) = \si_i^{\pm 1/2}(t_i) X_j^\pm X_i^\pm,\quad \forall i\neq j.
\end{equation}
\end{Remark}

\begin{Example}[The $n$:th Weyl algebra over $\K$]
Let $n\in\N$ be arbitrary and let $R=\K[u_1,u_2,\ldots,u_n]$. For $i=1,2,\ldots,n$, let $t_i=u_i$ and let $\sigma_i:R\to R$ be the unique $\K$-algebra automorphism determined by $\sigma_i(u_j)=u_j-\delta_{ij}$ where $\delta_{ij}$ is the Kronecker delta. Put $\Bt=(t_1,t_2,\ldots,t_n)$ and $\Bsi=(\sigma_1,\sigma_2,\ldots,\sigma_n)$.
Then $\CI$ is generated by the commutators $[X_i^\pm,X_j^\pm]$ and 
the TGWA $\CA(R,\Bsi,\Bt)$ is isomorphic to the $n$:th Weyl algebra $A_n(\K)$. 
\end{Example}

Bavula \cite{Bav1992} introduced higher rank generalized Weyl algebras. These are special cases of TGWAs (if the $t_i$ are regular).

\begin{Example}[Higher rank generalized Weyl algebras]
Let $n\in\N$ be arbitrary and let $R$ be any $\K$-algebra. For $i=1,2,\ldots,n$, let $t_i\in Z(R)$ be regular in $R$ and $\sigma_i:R\to R$ be $\K$-algebra automorphisms such that $\si_i(t_j)=t_j$ for all $i\neq j$. Put $\Bt=(t_1,t_2,\ldots,t_n)$ and $\Bsi=(\sigma_1,\sigma_2,\ldots,\sigma_n)$.
Then $\CI$ is generated by the commutators $[X_i^\pm,X_j^\pm]$ and 
the TGWA $\CA(R,\Bsi,\Bt)$ is isomorphic to the higher rank generalized Weyl algebra $R(\Bsi,\Bt)$.
\end{Example}

\begin{Example}[Primitive quotients of $U(\Fg)$]
Let $\Fg$ be a finite-dimensional complex simple Lie algebra with Cartan subalgebra $\Fh$.
Let $V$ be an infinite-dimensional simple weight $\Fg$-module such that $\dim V_\la\le 1$ for all $\la\in\Fh^\ast$. Let $J=\Ann_{U(\Fg)}V$ be the corresponding primitive ideal of the enveloping algebra $U(\Fg)$. Then $U(\Fg)/J$ is a TGWA \cite[Thm. 5.15]{HarSer2016}.
\end{Example}

For certain choices of $\Bsi$ and $\Bt$, relations \eqref{eq:TGWA-Rels} may be contradictory in the sense that $\tilde{\CA}(R,\Bsi,\Bt)=\{0\}$ (see \cite[Ex.~2.8]{FutHar2012a}). The following result resolves this problem when the $t_i$ are regular (i.e. non-zerodivisors) in $R$.
\begin{Theorem}[{\cite{FutHar2012a}}]
\label{thm:TGWA-Consistency}
If $t_1,t_2,\ldots, t_n$ are regular in $R$, then the following statements are equivalent:
\begin{enumerate}[{\rm (i)}]
\item The canonical homomorphism $R\to\tilde{\CA}(R,\Bsi,\Bt)$ is injective;
\item The canonical homomorphism $R\to\CA(R,\Bsi,\Bt)$ is injective;
\item The following equations hold, where we put $\si_i^\pm = \si_i^{\pm 1/2}$:
\begin{subequations} \label{eq:TGWA-Consistency-Rels}
\begin{align}
\label{eq:TGWA-Consistency-Rel1}
\si_j^{+}(t_i)\cdot \si_i^{+}(t_j)&=
\si_j^{-}(t_i)\cdot \si_i^{-}(t_j),\qquad \forall i\neq j,\\
\label{eq:TGWA-Consistency-Rel2}
\si_i^{+}\si_j^{+}(t_k)\cdot \si_i^{-}\si_j^{-}(t_k) &= \si_i^{+}\si_j^{-}(t_k)\cdot \si_i^{-}\si_j^{+}(t_k), \qquad \forall i\neq j\neq k\neq i.
\end{align}
\end{subequations}
\end{enumerate}
In particular, if Equations \eqref{eq:TGWA-Consistency-Rels} hold, then $\tilde{\CA}(R,\boldsymbol{\si},\boldsymbol{t})$ is non-trivial.
\end{Theorem}

We call Equations \eqref{eq:TGWA-Consistency-Rel1} (respectively \eqref{eq:TGWA-Consistency-Rel2}) the \emph{binary} (respectively \emph{ternary}) \emph{TGWA consisteny equations}.

\begin{Assumption}
In the rest of this paper we assume that $t_1,t_2,\ldots,t_n$ are regular in $R$, and that \eqref{eq:TGWA-Consistency-Rels} hold. Using the injective canonical homomorphisms $r\mapsto r1_{\tilde{\CA}}$ and $r\mapsto r1_{\CA}$ we identify $R$ with the degree zero components $\tilde{\CA}_0$ and $\CA_0$.
\end{Assumption}

\subsection{A characterization of the ideal $\CI$}
\label{sec:ideal-characterization} 
Finite generating sets for the ideal $\CI$ have been found in some special cases 
\cite{MazTur1999,Har2009,HarOin2013} but the existence of such sets is unknown in general. 
It is known that $\CI$ coincides with the gradation radical relative to certain Shapovalov-type forms \cite{MazPonTur2003,HarOin2013} and that $\CI$ is the sum of all graded left (or right) ideals intersecting $R$ trivially \cite[Cor.~5.4]{MazPonTur2003}.

In this section we give a new characterizations of $\CI$ which simplifies calculations and the construction of homomorphisms. It states that $\CI$ is simply the set of $R$-torsion elements in $\tilde{\CA}$.
Let $a\mapsto a^\ast$ denote the involution (anti-automorphism of order two) of $\tilde{\CA}$ given by $(X_i^\pm)^\ast=X_i^\mp$, $r^\ast=r \,\forall r\in R$.
Put
\[
\si_g=\si_1^{g_1}\circ\si_2^{g_2}\circ\cdots\circ\si_n^{g_n},\qquad \text{for $g=(g_1,g_2,\ldots,g_n)\in\Z^n$}.
\]
We need the following trace-like commutation relation.
\begin{Lemma}[{\cite[Lem.~5.1]{FutHar2012a}}]
If $a\in\tilde{\CA}$ is a \emph{monic monomial} i.e. of the form $a=X_{i_1}^{\ep_1}X_{i_2}^{\ep_2}\cdots X_{i_k}^{\ep_k}$ where $i_j\in\{1,2,\ldots,n\}$, $\ep_j\in\{+,-\}$ and if $b\in\tilde{\CA}_{-g}$ where $g=\deg(a)$, then
\begin{equation}\label{eq:TGWAtrace}
ab=\si_g(ba).
\end{equation}
\end{Lemma}
Let $R_\mathrm{reg}$ denote the set of regular elements of $R$. 

\begin{Theorem}\label{thm:TGWA-gradation-radical}
Let $\tilde{\CA}=\tilde{\CA}(R,\Bsi,\Bt)$ be a TGWC. Then the following subsets of $\tilde{\CA}$ coincide:
\begin{enumerate}[{\rm (i)}]
\item The sum $\CI$ of all graded ideals trivially intersecting $R$;
\item The set of $R$-torsion elements in $\tilde{\CA}$:
\begin{equation} \label{eq:TRA-def}
\begin{aligned}
T_R(\tilde{\CA})&=\{a\in \tilde{\CA}\mid \exists r\in R_\mathrm{reg}: r\cdot a=0\}\\
 &=\{a\in \tilde{\CA}\mid \exists r\in R_\mathrm{reg}: a\cdot r=0\}.
\end{aligned}
\end{equation}
\end{enumerate}
\end{Theorem}
\begin{proof}
We use the first equality in \eqref{eq:TRA-def} as the definition of $T_R(\tilde{\CA})$. The other case follows by applying the involution $\ast$ and using that $\CI^\ast=\CI$.

$\CI\subseteq T_R(\tilde{\CA})$: Let $a\in\CI$. Since $\CI$ is graded we may assume that $a$ is homogeneous. Let $g=\deg(a)$. Let $b$ be a product of generators $X_i^\pm$ such that $\deg(b)=-g$. Then $b^\ast b$ can be simplified using \eqref{eq:TGWA-Rels} to a product of elements of the form $\si_i^k(t_j)$ and is hence a non-zerodivisor in $R$. On the other hand, $ba\in\CI\cap R=\{0\}$, hence $(b^\ast b)a=0$ which proves that $a\in T_R(\tilde{\CA})$. 

$T_R(\tilde{\CA})\subseteq\CI$: Let $a\in T_R(\tilde{\CA})$. Again we may assume that $a$ is homogeneous and let $g=\deg(a)$. Thus $ra=0$ for some $r\in R_\mathrm{reg}$. 
We show that $\langle a\rangle \cap R=\{0\}$. It suffices to show that $bac=0$ for any monic monomials $b$ and $c$ such that $\deg(b)+\deg(a)+\deg(c)=0$. By \eqref{eq:TGWAtrace} we have
$bac=\si_h(acb)$ where $h=\deg(b)$. Since $ra=0$, clearly $racb=0$. Since $r$ is regular in $R$ and $acb\in R$, we conclude $acb=0$.
\end{proof}

The new description (ii) is surprising in that it only depends on $R$, not on the rest of the $\Z^n$-gradation of $\tilde{\CA}$. Two immediate applications are to the construction of homomorphisms and isomorphisms involving TGWAs.
\begin{Corollary}
If $\CB$ is any $\K$-algebra and $\tilde{\varphi}:\tilde{\CA}(R,\Bsi,\Bt)\to \CB$ is a $\K$-algebra homomorphism such that $\tilde{\varphi}(r)$ is regular in $\CB$ for all $r\in R_\mathrm{reg}$, then $\tilde{\varphi}$ induces a $\K$-algebra homomorphism $\varphi:\CA(R,\Bsi,\Bt)\to \CB$.
\end{Corollary}
\begin{Corollary}\label{cor:R-ring-Isomorphism}
If two TGWCs $\tilde{\CA}(R,\Bsi,\Bt)$ and $\tilde{\CA}(R,\Bsi',\Bt')$ (with the same base ring $R$) are isomorphic as $R$-rings, then so are the corresponding TGWAs $\CA(R,\Bsi,\Bt)$ and $\CA(R,\Bsi',\Bt')$.
\end{Corollary}

\begin{Corollary}\label{cor:fiber-TGWA}
Any noncommutative Kleinian fiber product $\CA_{\al_1,\al_2}(p_1,p_2)$ is isomorphic to a TGWA $\CA(R,\Bsi,\Bt)$ where $R=\C[u]$, $\Bsi=(\si_1,\si_2)$, $\si_i(u)=u-\al_i$, $\Bt=(p_1,p_2)$.
\end{Corollary}

\subsection{TGWA vs CGR}
The class of crystalline graded rings (CGR) was introduced by Nauwelaerts and Van Oystaeyen in \cite{NauVan2008} and further studied in \cite{NeiVan2009,NeiVan2010,OinSil2009}. Any generalized Weyl algebra is a CGR \cite[Sec.~3.2.6]{NauVan2008}. However, the same is not expected to be true for all TGWAs, due to the fact that $X_i^+$ and $X_j^+$ in general do not commute. Nevertheless, in this section we prove that if $R$ is a PID, then any TGWA with base ring $R$ is a crystalline graded ring. This will be applied to the computation of the center of the algebras $\CA(\mathscr{L})$ in Section \ref{sec:center}.

\begin{Definition}[CGR \cite{NauVan2008}]
A group-graded ring $D=\bigoplus_{g\in \Ga} D_g$ is called a 
\emph{crystalline graded ring} if there exist $a_g\in D_g$ such that
$D_g$ is free as a left and as a right $D_e$-module of rank one with basis $a_g$ for each $g\in \Ga$, where $e$ is the identity of the group $\Ga$.
\end{Definition}

\begin{Definition} Let $D$ be a ring and $M$ be a $D$-module. We say that
$M$ is \emph{uniform} if the intersection of any two nonzero submodules of $M$ is nonzero.
\end{Definition}

\begin{Lemma}\label{lem:PIDCGR}
 Let $\CA=\CA(R,\Bsi,\Bt)$ be a TGWA and let $g\in G$. Then
\begin{enumerate}[{\rm (i)}]
\item \label{it:PIDCGR-fg} $\CA_g$ is finitely generated as a left and as a right $R$-module;
\item \label{it:PIDCGR-ff} $\CA_g$ is faithful as a left and as a right $R$-module;
\item \label{it:PIDCGR-unif} If $R$ is uniform as a left (respectively right) $R$-module, then $\CA_g$ is uniform as a left (respectively right) $R$-module.
\end{enumerate}
\end{Lemma}
\begin{proof}
Statement \eqref{it:PIDCGR-fg} was proved in \cite[Cor.~3.3]{Har2006}.
For \eqref{it:PIDCGR-ff}, if $r\CA_g=0$ then in particular $rb_gb_g^\ast=0$ where $b_g\in\CA_g$ is a product of generators $X_i^\pm$. Since $b_gb_g^\ast$ can be simplified to a product of elements of the form $\si_1^{d_1}\cdots\si_n^{d_n}(t_i)$ which are regular in $R$, this implies $r=0$. Similarly $\CA_gr=0\Rightarrow r=0$.

To prove \eqref{it:PIDCGR-unif}, suppose $R$ is uniform as a left $R$-module. Write $g=(g_1,g_2,\ldots,g_n)$ and let $M_g=RX_1^{g_1}X_2^{g_2}\cdots X_n^{g_n}\subseteq \CA_g$ where we put  $X_i^k=(X_i^{\sgn(k)})^{|k|}$. Since $t_i\in R$ is regular for all $i$, $X_1^{g_1}X_2^{g_2}\cdots X_n^{g_n}$ is regular in $\CA$, hence $M_g\simeq R$ as a left $R$-module. In particular $M_g$ is uniform. Thus it suffices to show that any $R$-submodule of $\CA_g$ has nonzero intersection with $M_g$.
Let $a_g\in \CA_g$ be any nonzero element. We must show that $Ra_g\cap M_g\neq 0$.
Write $a_g=b_1+b_2+\cdots+b_s$ where each $b_i$ is a (nonzero) monomial of the form 
 $r X_{i_1}^{\ep_1}X_{i_2}^{\ep_2}\cdots X_{i_k}^{\ep_k}$
where $r\in R$, $i_j\in\{1,2,\ldots,n\}$ and $\ep_j\in\{\pm\}$.
Using the commutation relations \eqref{eq:TGWA-Rels} and the exchange relation \eqref{eq:TGWA-exchange-relation}, the generators $X_j^\pm$ in each $b_i$ can be rearranged, provided we multiply by a suitable product $p_i$ of elements of the form $\si_1^{d_1}\cdots\si_n^{d_n}(t_e)$. Thus for each such monomial $b_i$ there is a regular element $p_i\in R$ such that $p_ib_i \in M_g$. Let $p=p_1p_2\cdots p_s$. Then $pa_g\in M_g$, and $pa_g\neq 0$ since $p$ is regular. Thus $Ra_g\cap M_g\neq 0$ as required.
\end{proof}

\begin{Theorem} \label{thm:PIDCGR}
Let $\CA=\CA(R,\Bsi,\Bt)$ be a TGWA where $R$ is a PID. Then $\CA$ is a CGR.
\end{Theorem}
\begin{proof}
Let $g\in G$. By Lemma \ref{lem:PIDCGR}\eqref{it:PIDCGR-fg}, $\CA_g$ is finitely generated as a left $R$-module.
Thus, since $R$ is a PID,
$\CA_g\simeq R/(f_1)\oplus R/(f_2)\oplus\cdots\oplus R/(f_k)$
for some $f_i\in R$.
By Lemma \ref{lem:PIDCGR}\eqref{it:PIDCGR-unif}, $\CA_g$ is uniform as a left $R$-module, hence $k=1$.
By Lemma \ref{lem:PIDCGR}\eqref{it:PIDCGR-ff}, $f_1=0$. Hence $\CA_g\simeq R$ as a left $R$-module. Since $a_g r = \si_g(r)a_g$ for all $r\in R$, $a_g\in \CA_g$, a basis element for $\CA_g$ as a left $R$-module will also be a basis element for $\CA_g$ as a right $R$-module. Thus $\CA$ is a CGR.
\end{proof}

\begin{Corollary}\label{cor:PIDCGR}
Any noncommutative Kleinian fiber product $\CA_{\al_1,\al_2}(p_1,p_2)$ is a CGR.
\end{Corollary}

\begin{Example}\label{ex:PIDCGR}
Let $(m,n)=(2,1)$ and let $\mathscr{L}$ be as in Figure \ref{fig:21-single} where $\la=0$. Choose $(\al_1,\al_2)=(-1,2)$.
Then
\[P^\mathscr{L}_1(u)=u-\tfrac{-1}{2}\al_1=u-\tfrac{1}{2},\quad
P^\mathscr{L}_2(u)=(u-\tfrac{1}{2}\al_2)\big(u-(\al_1+\tfrac{1}{2}\al_2)\big)=(u-1)u.\]
Put $\CA=\CA(\mathscr{L})=\CA_{-1,2}\big(u-\tfrac{1}{2},\, (u-1)u\big)$ and identify $R=\C[H]$.
We find a generator $a_g$ for $\CA_g$ as a left and right $R$-module in the case when $g=(1,1)$. Take
\[a_g=\frac{1}{2}(X_1^+X_2^+-X_2^+X_1^+).\]
Using the exchange relation \eqref{eq:TGWA-exchange-relation} which in this instance can be written
\begin{equation}
X_2^+X_1^+(H+1)-X_1^+X_2^+(H-1)=0,
\end{equation}
one checks that
\[a_g \cdot (H+1) = X_1^+X_2^+,\quad a_g\cdot (H-1) = X_2^+X_1^+\]
Since $\{X_1^+X_2^+,\, X_2^+X_1^+\}$ generates $\CA_g$ as a right $R$-module, this proves that $\CA_g=a_g\cdot R$. As in the proof of Theorem \ref{thm:PIDCGR}, then automatically $\CA_g=R\cdot a_g$.
\end{Example}

\subsection{Rescaling isomorphisms in rank two}
In this subsection we show that for rank two TGWAs we can rescale the elements $t_i$ by central invertible elements from $R$, provided they solve the binary consistency relation \eqref{eq:TGWA-Consistency-Rel1}. 
This will be applied in the proof of the decomposition theorem for the category of weight modules, Theorem \ref{thm:A}.

\begin{Proposition} \label{prp:rank-2-isomorphisms}
Let $\CA=\CA(R,\Bsi,\Bt)$ and $\CA'=\CA(R,\Bsi,\Bt')$ be two TGWAs of rank two such that
$t_i'=s_i t_i$ for $i=1,2$, where $(s_1,s_2)$ is a pair of central invertible elements of $R$ that solve Equation \eqref{eq:TGWA-Consistency-Rel1} with respect to $(\si_1^{1/2},\si_2^{1/2})$.
Then $\CA\simeq \CA'$ as $\Z^2$-graded $R$-rings.
\end{Proposition}
\begin{proof}
Let $\tilde{\CA}=\tilde{\CA}(R,\Bsi,\Bt)$ and $\tilde{\CA}'=\tilde{\CA}(R,\Bsi,\Bt')$ be the corresponding TGWCs. We show that there exists a $\K$-algebra homomorphism
\[\varphi:\tilde{\CA}'\to\tilde{\CA}\]
determined by
\begin{gather*}
  \varphi(r) =r\quad\forall r\in R,    \\ 
\begin{alignedat}{2}
\varphi(X_1^+)&=\si_1^{1/2}(s_1) X_1^+, &\qquad 
\varphi(X_1^-)&=X_1^-,\\
\varphi(X_2^+)&=X_2^+,    &\qquad 
\varphi(X_2^-)&=\si_2^{-1/2}(s_2) X_2^-.
\end{alignedat}
\end{gather*}
For this, we need to show that the defining relations for $\tilde{\CA}'$ are preserved.
In all four cases $i\in\{1,2\}$ and $\pm\in\{+,-\}$ one checks that we have 
\[\varphi(X_i^\pm)\varphi(X_i^\mp)= \varphi\big(\si_i^{\pm 1/2} ( t_i') \big).\]
In addition, the fact that we placed the $s_i$ asymmetrically means that
\[ [\varphi(X_1^-),\varphi(X_2^+)]= 0\]
trivially, while
\[ [\varphi(X_1^+),\varphi(X_2^-)]=\si_1^{1/2}\si_2^{-1/2}\big(\si_2^{1/2}(s_1)\si_1^{1/2}(s_2)-\si_2^{-1/2}(s_1)\si_1^{-1/2}(s_2)\big)\cdot X_1^+X_2^-\]
which is zero exactly because $(s_1,s_2)$ solve the binary consistency equation \eqref{eq:TGWA-Consistency-Rel1}.
The final relation is easy to check since the $s_i$ are central in $R$:
\[\varphi(X_i^\pm)\varphi(r)=\varphi\big(\si_i^{\pm 1}(r)\big)\varphi(X_i^\pm). \] 
This proves that $\varphi$ is well-defined. Since $s_i$ are invertible in $R$ it is easy to see that $\varphi$ is an isomorphism. Furthermore, it is an isomorphism of $R$-rings, since $\varphi(r)=r$ for $r\in R$. Thus by Corollary \ref{cor:R-ring-Isomorphism}, $\CA\simeq \CA'$ as $R$-rings. Since the $s_i$ have degree zero, this is also an isomorphism of $\Z^2$-graded algebras.
\end{proof}

\subsection{Weight modules without inner breaks}
\label{sec:no-inner-breaks}

Let $\CA=\CA(R,\boldsymbol{\si},\mathbf{t})$ be a TGWA. From now on we assume that $R$ is commutative. Let $\Specm(R)$ be the set of maximal ideals of $R$.

In this section we first show that in the Levi-type decomposition $\CC_\Fm=\CT_\Fm\oplus \CJ_\Fm$ of certain subquotients $\CC_\Fm$ of $\CA$ as given in \cite{Har2006}, the subspace $\CJ_\Fm$ coincides with the gradation radical of $\CC_\Fm$. Secondly we give a new characterization of weight modules without inner breaks as those whose weight spaces are annihilated by $\CJ_\Fm$. Lastly, we recall and clarify the classification of simple weight modules without inner breaks from \cite{Har2006}. 

\begin{Definition}
An $\CA$-module $M$ is called a \emph{weight module} if
\begin{equation}
M=\bigoplus_{\Fm\in \Specm(R)} M_\Fm,\qquad M_\Fm=\{v\in M\mid \Fm v=0\}.
\end{equation}
An element $\Fm\in\Specm(R)$ is called a \emph{weight} for $M$ if $M_\Fm\neq\{0\}$. The set of weights for $M$ is called the \emph{support} of $M$ and is denoted $\Supp(M)$.
\end{Definition}

By \eqref{eq:TGWA-Rels}, we have $X_i^\pm M_\Fm\subseteq M_{\si_i^{\pm 1}(\Fm)}$ which generalizes to
\begin{equation}\label{eq:AgMm}
\CA_g M_\Fm \subseteq M_{\si_g(\Fm)}
\end{equation}
where 
$\si_g = \si_1^{g_1}\circ\si_2^{g_2}\circ\cdots\circ\si_n^{g_n}$ for $g=(g_1,g_2,\ldots,g_n)\in G=\Z^n$.
For $\Fm\in\Specm(R)$, let $G_\Fm=\Stab_{G}(\Fm)$ denote the corresponding stabilizer subgroup of $G$, consisting of all $g\in G$ such that $\si_g(\Fm)=\Fm$.
Since $G$ is abelian, $G_\Fm=G_\Fn$ if $\Fm$ and $\Fn$ belong to the same $G$-orbit in $\Specm(R)$.
\begin{Definition}\label{def:cyclic-subalgebra}
The \emph{cyclic subalgebra of $\CA$ at $\Fm$} is defined as
\begin{equation}\label{eq:Cm-def}
\mathcal{C}(\Fm)=\bigoplus_{g\in G_\Fm} \CA_g.
\end{equation}
By definition, $\mathcal{C}(\Fm)$ is a graded algebra with gradation group $G_\Fm$, and $R=\CA_0\subseteq \CC(\Fm)$.
The \emph{cyclic subquotient of $\CA$ at $\Fm$} is defined to be the quotient algebra
\begin{equation}
\CC_\Fm = \CC(\Fm)/\CC(\Fm)\Fm\CC(\Fm),
\end{equation}
which inherits a $G_\Fm$-gradation since $\Fm\subseteq \CC(\Fm)_0$.
\end{Definition}
The terminology comes from Lie theory, where the centralizer of a Cartan subalgebra, $U(\Fg)^{\Fh}$, is called the cyclic subalgebra of $U(\Fg)$.
Since $\si_g(\Fm)=\Fm$ for all $g\in G_\Fm$, we have $\CC(\Fm)\Fm=\CC(\Fm)\Fm \CC(\Fm)=\Fm \CC(\Fm)$.
Thus, by \eqref{eq:AgMm} and that $\Fm M_\Fm=0$, every weight space $M_\Fm$ is a $\CC_\Fm$-module.

Let
\begin{equation} \label{eq:Hm-Def}
H_\Fm= \big\{g\in G_\Fm\;|\; \exists a\in \CA_g: a^\ast\cdot a \notin \Fm\big\}.
\end{equation}
As shown in \cite{Har2006}, $H_\Fm$ is a subgroup of $G_\Fm$. Thus we have a corresponding decomposition
\begin{equation} \label{eq:Cm-Decomposition}
 \CC_\Fm = \CT_\Fm \oplus \CJ_\Fm,\qquad 
\CT_\Fm = \bigoplus_{g\in H_\Fm} (\CC_\Fm)_g,\qquad 
\CJ_\Fm = \bigoplus_{g\in G_\Fm, g\notin H_\Fm} (\CC_\Fm)_g,
\end{equation}
where $\CT_\Fm$ is a subalgebra of $\CC_\Fm$, and $\CJ_\Fm$ is a subspace that will turn out to be an ideal.
To state the next theorem we need a definition.
\begin{Definition}\label{def:gradation-radical}
Let $\Ga$ be a group and $e\in\Ga$ the identity.
The \emph{gradation form} 
of a $\Ga$-graded algebra $D=\bigoplus_{g\in \Ga} D_g$ is the $\Z$-bilinear map
$\ga:D\times D\to D_e$ given by
\begin{equation}
\ga(x,y)=\mathrm{Proj}_e (xy), \qquad x,y\in D,
\end{equation}
where $\mathrm{Proj}_e:D\to D_e$, $\sum_g x_g\mapsto x_e$ is the projection onto $D_e$ relative
to the direct sum $D=\bigoplus_{g\in \Ga} D_g$.
The \emph{gradation radical} of $D$ is the ideal
\begin{equation}
\mathrm{gRad}(D)=\big\{x\in D\mid \forall y\in D: \ga(x,y)=0=\ga(y,x)\big\}.
\end{equation}
\end{Definition}

\begin{Theorem} \label{thm:TmJm}
Let $\CA=\CA(R,\Bsi,\Bt)$ be a TGWA and let $\Fm\in\Specm(R)$.
Then
\begin{enumerate}[{\rm (a)}]
\item $\CT_\Fm$ is a crossed-product algebra over $R/\Fm$ with respect to the group $H_\Fm$,
\item $\CJ_\Fm$ coincides with the gradation radical of the $G_\Fm$-graded algebra $\CC_\Fm$.
\end{enumerate}
\end{Theorem}
\begin{proof}
(a) Each graded component of $\CT_\Fm$ contains an invertible element by definition of $H_\Fm$, hence $\CT_\Fm$ is a crossed-product algebra. See also \cite[Thms.~4.5,4.8]{Har2006}.

(b) First we show that if $g_1,g_2\in G_\Fm\setminus H_\Fm$ are such that $g_1+g_2\in H_\Fm$, then $\CA_{g_1}A_{g_2}\in \Fm \CA$, and consequently $(\CC_\Fm)_{g_1} (\CC_\Fm)_{g_2} = 0$.
For $i=1,2$, let $b_i\in \CA_{g_i}$. Then $b_1b_2\in \CA_h$, where $h=g_1+g_2\in H_\Fm$. Thus $b_1b_2=r a_h$ (mod $\CA\Fm$) for some $r\in R$ and some $a_h\in \CA_h$ with $a_h' a_h\notin \Fm$ for some $a_h'\in \CA_{-h}$.
We have
$a_h' b_1 b_2 = (a_h' b_1) b_2 \in \Fm $ since $a_h' b_1\in \CA_{-g_2}$ and $g_2\notin H_\Fm$. On the other hand, $a_h' b_1b_2= a_h'  r a_h=\si_{-h}(r)a_h' a_h$. Since $a_h' a_h\notin\Fm$ and $\Fm$ is maximal hence prime, $r\in\si_h(\Fm)=\Fm$. Thus $b_1b_2\in \Fm \CA$.

Now it follows directly that $\CJ_\Fm$ is an ideal and that $\CJ_\Fm\subseteq \mathrm{gRad}(\CC_\Fm)$.
Conversely, suppose that $b\in \mathrm{gRad}(\CC_\Fm)$. Since the gradation radical is a graded subspace of $\CC_\Fm$, we may without loss of generality assume that $b\in (\CC_\Fm)_g$ for some $g\in G_\Fm$. Since $b$ is in the gradation radical, $b' b=0$ for all $b'\in (\CC_\Fm)_{-g}$. By definition of $H_\Fm$ this implies that $g\notin H_\Fm$. Thus $b\in \CJ_\Fm$.
\end{proof}

Next we turn to weight modules without inner breaks. These were first introduced, under some  restrictions on the support, in \cite{MazTur1999,MazPonTur2003} and without restrictions in \cite{Har2006}.
Furthermore, in \cite{Har2006}, the simple weight modules without inner breaks were classified. 
We show that the class of weight modules without inner breaks can be naturally be described as those simple weight modules
where each weight space $M_\Fm$ is annihilated by the gradation radical of the cyclic subquotient $\CC_\Fm$.
This gives a different way to understand the notion of inner breaks. 

\begin{Definition}[\cite{Har2006}]
Let $M$ be a simple weight module over a TGWA $\CA$.
We say that $M$ has \emph{no inner breaks}
if for any $\Fm\in\Supp(M)$ and any monomial $a=rX_{i_1}^{\ep_1}X_{i_2}^{\ep_2}\cdots X_{i_k}^{\ep_k}\in \CA$
(where $r\in R, i_j\in\{1,2,\ldots,n\},\ep_j\in\{\pm\}$) such that $aM_\Fm\neq 0$
 we have $a^\ast\cdot  a\notin\Fm$.
\end{Definition}

The following lemma shows that the requirement that $a$
should be a monomial can be relaxed to being homogenous (with respect to the $G$-gradation on $\CA$).
\begin{Lemma} \label{lem:NIBequiv}
 Let $M$ be a simple weight module over $\CA$.
Then $M$ has no inner breaks
iff for any $\Fm\in\Supp(M)$ and any homogenous $a\in \CA$
with $aM_\Fm\neq 0$ we have $a^\ast\cdot a\notin\Fm$.
\end{Lemma}
\begin{proof} ($\Leftarrow$): Trivial because any monomial is homogenous.
($\Rightarrow$): Assume $M$ has no inner breaks.
Suppose $\Fm\in\Supp(M)$, and $a\in \CA$ is a homogenous
element with $aM_\Fm\neq 0$. Let $g\in \Z^n$ be the degree
of $a$. Write $a$ as a sum $a_1+\cdots+a_s$ of monomials
of degree $g$. Then $a_iM_\Fm\neq 0$ for at least one $i$.
Since $M$ has no inner breaks, $a_i^\ast a_i\notin\Fm$. So $a_i^\ast a_i$ acts
bijectively on $M_\Fm$. Thus
\[0\neq aM_\Fm = a a_i^\ast a_i M_\Fm \subseteq a a_i^\ast M_{\si_g(\Fm)}.\]
Thus $a a_i^\ast$, which has degree zero and thus lies in $R$,
does not lie in the maximal ideal $\si_g(\Fm)$. So $\si_g(\Fm)\not\ni (aa_i^\ast)^2=(aa_i^\ast)^\ast(aa_i^\ast)
=a_i(a^\ast a)a_i^\ast = \si_g(a^\ast a)\cdot a_ia_i^\ast$
which implies that $a^\ast a\notin \Fm$.
\end{proof}

\begin{Theorem}\label{thm:no-inner-breaks}
Let $\CA=\CA(R,\Bsi,\Bt)$ be a TGWA. For each maximal ideal $\Fm$ of $R$, let $\CJ_\Fm$ be the gradation radical of the cyclic subquotient $\CC_\Fm$ of $\CA$. Let $M$ be a simple weight $\CA$-module.
Then the following two statements are equivalent.
\begin{enumerate}[{\rm (i)}]
\item $M$ has no inner breaks;
\item $\CJ_\Fm M_\Fm=0$ for all $\Fm\in\Specm(R)$.
\end{enumerate}
\end{Theorem}

\begin{proof} (i)$\Rightarrow$(ii): Proved in \cite[Thm.~4.5(a)]{Har2006}.

(ii)$\Rightarrow$(i):
Let $\Fm\in\Supp(M)$ and assume that $a\in\CA$ is a homogeneous element such that $aM_\Fm\neq 0$.
We must show that $a^\ast \cdot a\notin\Fm$. Since $M$ is simple, there exists a homogeneous element $b\in \CA$ such that $baM_\Fm$ is a nonzero subspace of $M_\Fm$ (otherwise $aM_\Fm$ would generate a nonzero proper $\CA$-submodule of $M$). Since $baM_\Fm$ is a nonzero subspace of $M_\Fm$, while on the other hand $baM_\Fm\subseteq M_{\si_g(\Fm)}$ where $g=\deg(ba)$, it follows that $g\in G_\Fm$. This shows that $ba$ belongs to the cyclic subalgebra $\CC$ of $\CA$. By (ii) and that $baM_\Fm\neq 0$, the image in $\CC_\Fm$ of $ba$ does not belong to $\CJ_\Fm$. Therefore $g\in H_\Fm$ by \eqref{eq:Cm-Decomposition}. By the definition \eqref{eq:Hm-Def} of $H_\Fm$ we conclude that $(ba)^\ast\cdot (ba)\notin \Fm$. Since $(ba)^\ast\cdot (ba)=a^\ast b^\ast b a = \si_h(b^\ast b) a^\ast a$ where $h=\deg(a^\ast)$, we conclude that $a^\ast a\notin\Fm$.
\end{proof}

We end by clarifying the classification of simple weight modules without inner breaks from \cite{Har2006}.

\begin{Lemma} \label{lem:partition}
Let $\CA=\CA(R,\Bsi,\Bt)$ be a TGWA. Let
\[\mathscr{S}=\big\{\Supp(V)\mid \text{$V$ is a simple weight $\CA$-module without inner breaks}\big\}.\]
Then the following three statements hold.
\begin{enumerate}[{\rm (a)}]
\item If $S\in\mathscr{S}$ and $\Fm,\Fn\in S$, then $\CC_\Fm\simeq\CC_\Fn$ as graded $\K$-algebras.
\item $\mathscr{S}$ is is a partition of $\Specm(R)$.
\item Put $S_\Fm=\{\si_g(\Fm)\mid \text{$g\in \Z^n$ such that $a^\ast \cdot a\notin \Fm$ for some $a\in \CA_g$}\big\}$. Then
\begin{equation}
\mathscr{S}=\{S_\Fm\mid \Fm\in\Specm(R)\}.
\end{equation}
\end{enumerate}
\end{Lemma}
\begin{proof}
(a) By Lemma \ref{lem:NIBequiv} there exists $g\in\Z^n$ and $a\in \CA_g$ with $\si_g(\Fm)=\Fn$ and $a^\ast\cdot a\notin\Fm$. Since $\Fm$ is maximal there is $a'\in \CA_{-g}$ with $a'\cdot a\in 1+\Fm$ and by \eqref{eq:TGWAtrace} $a\cdot a'=\si_g(a'\cdot a)\in 1+\Fn$. Define $\varphi_a: \CC\to\CC$ by $\varphi_a(b)=aba'$. Clearly $\varphi_a(\Fm)\subseteq \Fn$, hence we get a map $\tilde\varphi_a:\CC_\Fm\to\CC_\Fn$. It is straightforward to check that the latter is a graded isomorphism.

(b) First we show $\cup\mathscr{S}=\Specm(R)$. Let $\Fm\in \Specm(R)$ and $N$ be a simple $\CT_\Fm$-module and extend it to a (simple) $\CC_\Fm$-module by requiring $\CJ_\Fm N=0$. By \cite[Prop.~7.2]{MazPonTur2003}, the induced module $\CA\otimes_{\CC} N$ has a unique simple quotient $M$ such that $M_\Fm\simeq N$ as $\CC_\Fm$-modules.
For any other $\Fn\in\Supp(M)$, let $\varphi:\CC_\Fm\to\CC_\Fn$ be a graded isomorphism as constructed in part (a). In particular $\varphi(\CJ_\Fm)=\CJ_\Fn$. Then $\CJ_\Fn M_\Fn=\varphi(\CJ_\Fm)M_\Fn=a\CJ_{\Fm}a'M_\Fn\subseteq a \CJ_\Fm M_\Fm = 0$. By Theorem \ref{thm:no-inner-breaks}, $M$ has no inner breaks. This proves $\Fm\in\cup\mathscr{S}$ as desired.

Next, suppose $\Supp(V)\cap\Supp(W)\neq\emptyset$ for some simple weight $\CA$-modules $V$ and $W$ without inner breaks. We claim that $\Supp(V)=\Supp(W)$. By symmetry it suffices to show that $\Supp(V)\subseteq\Supp(W)$. Let $\Fm\in\Supp(V)\cap\Supp(W)$ and $\Fn\in\Supp(V)$. Since $V$ is simple there exists $a\in\CA$ such that $aV_\Fm\neq 0$ and $aV_\Fm\subseteq V_\Fn$. Writing $a$ as a sum of homogeneous elements, $a=a_1+a_2+\cdots+a_k$ where $g_i=\deg(a_i)\in\Z^n$, there exists at least one term $a_i$ such that $a_iV_\Fm\neq 0$ and $a_iV_\Fm\subseteq V_\Fn$. Therefore we may assume that $a$ itself is homogeneous to begin with. Let $g=\deg(a)$. Since $V$ has no inner breaks, Lemma \ref{lem:NIBequiv} implies that $a^\ast\cdot a\notin\Fm$. Therefore $a^\ast \cdot a$ acts invertibly on the weight space $W_\Fm$. In particular $aW_\Fm\neq 0$. Since $\deg(a)=g$ and $\si_g(\Fm)=\Fn$, we have $aW_\Fm\subseteq W_\Fn$. Thus $\Fn\in\Supp(W)$.

(c) If $M$ is any simple weight $\CA$-module without inner breaks, then \cite[Cor.~5.2]{Har2006} implies that
$\Supp(M)=S_\Fm$ for any $\Fm\in\Supp(M)$.
\end{proof}

For each $S\in \mathscr{S}$, pick a maximal ideal $\Fm(S)\in S$ at random.

\begin{Theorem}[{\cite{Har2006}}] \label{thm:TGWA-modules-classification}
There is a bijective correspondence between the set of isoclasses of simple weight $\CA$-modules without inner breaks, and the set of pairs $(S,N)$ where $S\in\mathscr{S}$ and $N$ is an isoclass of simple $\CC_{\Fm(S)}$-modules.
\end{Theorem}

\section{Proof of Theorem \ref{thm:A}}\label{sec:ProofA}
Recall from Definition \ref{def:NCKFP} the notion of a noncommutative Kleinian fiber product $\CA_{\al_1,\al_2}(p_1,p_2)$. By Corollary \ref{cor:fiber-TGWA}, $\CA_{\al_1,\al_2}(p_1,p_2)$ is isomorphic to the TGWA $\CA(R,\Bsi,\Bt)$ where $R=\C[u]$, $\Bsi=(\si_1,\si_2)$, $\si_i(u)=u-\al_i$ and $\Bt=(p_1(u),p_2(u))$.

\begin{Lemma}\label{lem:fiber-isomorphisms}
We have the following isomorphisms.
\begin{enumerate}[{\rm (i)}]
\item \label{it:fiber-transposition}
(Transposition)
$\CA_{\al_1,\al_2}(p_1,p_2)\simeq \CA_{\al_2,\al_1}(p_2,p_1)$.
\item (Affine transformations)
If $\psi(u)=\ga u+\ga_0$ for some $\ga\in\C\setminus\{0\}, \ga_0\in\C$ then
\begin{equation}\label{eq:affine-isomorphism}
\CA_{\al_1,\al_2}(p_1\circ\psi,p_2\circ\psi)\simeq\CA_{\ga\al_1,\ga\al_2}(p_1,p_2).
\end{equation}
\end{enumerate}
\end{Lemma}
\begin{proof}
(i) Obvious.

(ii) Let $q_i=p_i\circ \psi$.
We construct a homomorphism $\Psi:\CA_{\ga\al_1,\ga\al_2}(p_1,p_2)\to \CA_{\al_1,\al_2}(q_1,q_2)$ as follows.
Let $\Psi_F$ be the unique homomorphism from the free algebra on the set 
\[\{X_1^+,X_1^-,X_2^+,X_2^-,H\}\]
to $\CA_{\al_1,\al_2}(q_1,q_2)$ determined by
 $\Psi_F(X_i^\pm)=X_i^\pm$, $\Psi_F(H)=\ga H+ \ga_0$.
It is straightforward to verify that the defining relations \eqref{eq:Aalbepq-rels} for $\tilde{\CA}_{\ga\al_1,\ga\al_2}(p_1,p_2)$ are preserved. For example,
\[\Psi_F\big(X_i^\pm X_i^\mp-p_i(H\mp \ga\al_i/2)\big)=X_i^\pm X_i^\mp-p_i(\ga H+\ga_0\mp\ga\al_i/2)=X_i^\pm X_i^\mp-q_i(H\mp\al_i/2)=0.\]
Hence $\Psi_F$ induces a homomorphism $\tilde{\Psi}:\tilde{\CA}_{\ga\al_1,\ga\al_2}(p_1,p_2)\to\tilde{\CA}_{\al_1,\al_2}(q_1,q_2)$. If $a\in\tilde{\CA}_{\ga\al_1,\ga\al_2}(p_1,p_2)$ is such that $f(H)\cdot a=0$ for some polynomial $f$, then applying $\tilde{\Psi}$ we get $f(\ga H+\ga_0)\cdot \tilde\Psi(a)=0$, which shows that $\tilde{\Psi}(\CI)\subseteq \CI$. Hence $\tilde{\Psi}$ induces an homomorphism $\Psi$ between the corresponding quotients as required. Since $\psi$ is invertible, so is $\Psi$.
\end{proof}

We now prove our first main theorem, Theorem \ref{thm:A}, stated in Section \ref{sec:thmA}.

\begin{proof}[Proof of Theorem \ref{thm:A}]
Since $\CA$ is non-trivial, by Proposition \ref{prp:NKFP-Consistency} and Theorem \ref{thm:HarRos-classification}, 
$\Z\al_1+\Z\al_2$ has rank one, hence $\Z\al_1+\Z\al_2=\Z\ga$ for some $\ga\in \C^\times$.
After applying the affine transformation \eqref{eq:affine-isomorphism} corresponding to $\psi(u)=\ga^{-1}u$, we may assume that $\Z\al_1+\Z\al_2=\Z$.

By the defining relations \eqref{eq:NCKS-rels} of $\CA_{\al_1,\al_2}(p_1,p_2)$, if $Hv=\la v$ then $HX_i^\pm v=(\la\pm\al_i)X_i^\pm v$ and hence the category of weight modules is a direct product of subcategories consisting of all modules with support contained in a single coset $\om\in\C/(\Z\al_1+\Z\al_2)=\C/\Z$:
\[\mathscr{W}_{\al_1,\al_2}(p_1,p_2)\simeq \prod_{\om\in\C/\Z} \mathscr{W}_{\al_1,\al_2}(p_1,p_2)_\om.\]
Each such subcategory $\mathscr{W}_{\al_1,\al_2}(p_1,p_2)_\om$
is equivalent to the category of all weight modules
over a corresponding localization of the algebra $\CA_{\al_1,\al_2}(p_1,p_2)$ at the multiplicative set generated by  $H-\mu$, $\mu\in\C\setminus\om$ (where $\setminus$ is set difference). By localization results in \cite{FutHar2012a}, the resulting algebra is isomorphic to the TGWA $\CA\big(R_\om,(\si_1,\si_2),(p_1,p_2)\big)$ where $R_\om=\{f/g\in\C(u)\mid \forall n\in\om:\, g(n)\neq 0\}$ and $\si_i(H)=H-\al_i$.

By Theorem \ref{thm:HarRos-classification}, we can write $p_i(u)=\tilde{p}_i(u) P_i^{\mathscr{L}^{(\om)}}(u-\la)$ for some $(m,n)$-periodic lattice configuration $\mathscr{L}^{(\om)}$ and $\la\in\om$. Here $\tilde{p}_i\in R_\om$ is the product of the remaining factors in \eqref{eq:HR-factorization}, all having zeros in cosets other than $\om$, hence $\tilde{p}_i$ is invertible in $R_\om$ for $i=1,2$. In addition,
$(\tilde{p}_1,\tilde{p}_2)$ also solves the Mazorchuk-Turowska equation \eqref{eq:MTeq} by Theorem \ref{thm:HarRos-classification}. Therefore, Proposition \ref{prp:rank-2-isomorphisms} implies that 
\[
\CA\big(R_\om,(\si_1,\si_2),(p_1,p_2)\big)\simeq \CA\big(R_\om,(\si_1,\si_2),(P_1^{\mathscr{L}^{(\om)}}(u-\la),P_2^{\mathscr{L}^{(\om)}}(u-\la))\big)
\]
as $R_\om$-rings (that is, $H\mapsto H$ under this isomorphism).
Using a translation isomorphism as in \eqref{eq:affine-isomorphism} associated to $\psi(u)=u-\la$,
\[\CA\big(R_\om,(\si_1,\si_2),(P_1^{\mathscr{L}^{(\om)}}(u-\la),P_2^{\mathscr{L}^{(\om)}}(u-\la))\big)
\simeq 
 \CA\big(R_\Z,(\si_1,\si_2),(P_1^{\mathscr{L}^{(\om)}},P_2^{\mathscr{L}^{(\om)}})\big).
\]  
By localization arguments as above, the category of weight modules over the right hand side is equivalent to $\mathscr{W}(\mathscr{L}^{(\om)})_\Z$.
\end{proof}

\section{Simple weight modules over $\CA(\mathscr{L})$}

Having reduced the problem of classifying simple weight modules over $\CA_{\al_1,\al_2}(p_1,p_2)$ to the 
problem of classifying simple integral weight modules over $\CA(\mathscr{L})=\CA_{-n,m}\big(P_1^\mathscr{L}(u;-n,m),P_2^{\mathscr{L}}(u;-n,m)\big)$, we focus our attention on the latter in this section,
with the goal to prove the second main theorem, Theorem \ref{thm:B}, from Section \ref{sec:thmB}.

\subsection{Notation}
We fix notation that will be used for the rest of the paper.
Let $(m,n)$ be a pair of relatively prime non-negative integers and
$(\al_1,\al_2)\in\C^2$ be a nonzero vector such that $m\al_1+n\al_2=0$. In examples or when otherwise compelled to make a choice we always choose $(\al_1,\al_2)=(-n,m)$. To simplify notation we put:
\begin{align*}
F &= \Z\al_1+\Z\al_2=\Z && \text{(face lattice)}\\ 
V &= F + (\al_1+\al_2)/2 && \text{(vertex lattice)}\\
E_i &= F+\al_i/2,\quad i=1,2 &&\text{(midpoints of vertical and horizontal edges respectively)}
\end{align*}
For $\la=x_1\al_1+x_2\al_2\in F$ we put $\bar\la=(x_1,x_2)+\Ga_{m,n}\in \T_{m,n}$. The map $\la\mapsto\bar\la$ is well-defined and injective. Put $\bar F=\{\bar \la\mid \la\in F\}$.

Let $\mathscr{L}=(\mathscr{L}_1,\mathscr{L}_2)$ be a pair of functions $\mathscr{L}_i:E_i\to\Z_{\ge 0}$ such that $\mathscr{L}_i(e)=0$ for all but finitely many $e\in E_i$,
and the ice rule holds (cf. Figure \ref{fig:vertex}) for all $v\in V$:
\begin{equation}\label{eq:ice-rule-alpha}
\mathscr{L}_1(v-\al_1/2)+\mathscr{L}_2(v-\al_2/2)=\mathscr{L}_1(v+\al_1/2)+\mathscr{L}_2(v+\al_2/2)
\end{equation}
The polynomials $P_i^\mathscr{L}(u)\in\C[u]$ from \eqref{eq:fundamental-solution} can then be written
\begin{equation}
P^\mathscr{L}_i(u)=\prod_{e\in E_i}(u-e)^{\mathscr{L}_i(e)}
\end{equation}
Let $\CA=\CA(\mathscr{L})=\CA_{\al_1,\al_2}(P_1^\mathscr{L},P_2^\mathscr{L})$ be the corresponding noncommutative Kleinian fiber product.

\subsection{Simple integral weight modules without inner breaks}
Since $\CA$ is a TGWA of rank two, it is a $\Z^2$-graded algebra, with $\deg(X_1^\pm)=(\pm 1,0)$, $\deg(X_2^\pm)=(0,\pm 1)$, $\deg(H)=(0,0)$.

\begin{Definition}[Cyclic subalgebra]
The \emph{cyclic subalgebra} of $\CA$ is the centralizer of $H$ in $\CA$:
\begin{equation}
\CC=C_\CA(H)=\{a\in \CA\mid [H,a]=0\}.
\end{equation}
\end{Definition}

We define the following cyclic subgroup of the gradation group $\Z^2$:
\begin{equation}
G_0=\langle(m,n)\rangle
\end{equation}

\begin{Lemma}
$\CC=\bigoplus_{g\in G_0} \CA_g$ and $\CC$ is a maximal commutative subalgebra of $\CA$.
\end{Lemma}
\begin{proof}
Since $H$ is homogeneous of degree zero, $\CC$ is a graded subalgebra of $\CA$. If $a\in \CA$ has degree $(d_1,d_2)\in\Z^2$ then defining relations of $\CA$ imply $[H,a]=(d_1\al_1+d_2\al_2)a$. This is zero if and only if $(d_1,d_2)$ is a multiple of $(m,n)$. By \cite[Cor.~5.4]{HarOin2013} the algebra $\CC$ is commutative, hence maximal commutative.
\end{proof}

In the language of TGWAs, this means that $\CC(\Fm)$ is independent of $\Fm\in\Specm(R)$, see Definition \ref{def:cyclic-subalgebra}.

\begin{Definition}[Cyclic subquotients]
Let $\la\in F$. The \emph{cyclic subquotient of $\CA$ at $\la$}, denoted $\CC_\la$, is defined as the quotient of $\CC$ by the principal ideal generated by $H-\la$:
\begin{equation}
\CC_\la=\CC/(H-\la)
\end{equation}
\end{Definition}
Since $\CC_\la$ is a quotient of $\CC$, it follows that $\CC_\la$ is also commutative, and since $H-\la$ is homogeneous with respect to the $G_0$-gradation, $\CC_\la$ inherits a gradation by the same group $G_0$.
Let $\CJ_\la$ denote the gradation radical of the $G_0$-graded algebra $\CC_\la$ (recall Definition \ref{def:gradation-radical}).

Next we need some notation involving binary sequences.
Let $\mathsf{Seq}_2$ denote the set of all finite sequences $\un{i}=i_1i_2\cdots i_\ell$ where $i_j\in\{1,2\}$ for each $j$. For $k\in\{1,2\}$, let $\ell_k(\un{i})=|\{j\in\{1,2,\ldots,\ell\}\mid i_j=k\}|$ be the number of times $k$ appear in the sequence $\un{i}$, and let $\ell(\un{i})=\ell_1(\un{i})+\ell_2(\un{i})=\ell$ denote the length of $\un{i}$. The empty sequence has length zero. Let $\mathsf{Seq}_2(m,n)$ be the subset of $\mathsf{Seq}_2$ consisting of all sequences $\un{i}$ (necessarily of length $m+n$) with $\ell_1(\un{i})=m$ and $\ell_2(\un{i})=n$.

\begin{Definition}[Based face and vertex paths]
Let $\un{i}\in\mathsf{Seq}_2$ and $\la\in F$. The \emph{based face path} $\bar\pi(\un{i},\la)\subset \T_{m,n}$ is defined to be the union of line segments
\[ \bar\pi(\un{i},\la)=[\bar \la, \bar \la+\mathbf{e}_{i_1}]\cup [\bar \la + \mathbf{e}_{i_1}, \bar \la + \mathbf{e}_{i_1}+\mathbf{e}_{i_2}] \cup \cdots \cup [\bar \la + \mathbf{e}_1+\mathbf{e}_2+\cdots+\mathbf{e}_{\ell-1}, \bar \la+\mathbf{e}_1+\mathbf{e}_2+\cdots\mathbf{e}_\ell],\]
where we used the natural translation action of $\Z^2$ on $\T_{m,n}=\R^2/\langle(m,n)\rangle$.
The definition of a \emph{based vertex path} $\bar\pi(\un{i},v)$ is obtained by replacing $\la\in F$ by $v\in V$. Similarly we define based face paths $\pi(\un{i},\la)$ in $\R^2$.
\end{Definition}

For $\un{i}=i_1i_2\cdots i_\ell\in\mathsf{Seq}_2$ we define
\begin{equation}
X(\un{i})=X_{i_\ell}^+X_{i_{\ell-1}}^+\cdots X_{i_1}^+ \in \mathcal{A}(\mathscr{L}).
\end{equation}
Recall from Section \ref{sec:ideal-characterization} that $\mathcal{A}(\mathscr{L})$, being a TGWA, carries
an anti-automorphism $\ast$ of order two, uniquely determined by
\[H^\ast=H, \qquad (X_i^\pm)^\ast=X_i^\mp, \qquad (a+b)^\ast= a^\ast+b^\ast,\qquad (ab)^\ast=b^\ast a^\ast,\qquad \mu^\ast = \mu\]
for $i=1,2$ and all $a,b\in\mathcal{A}(\mathscr{L})$, $\mu\in \C$. In particular, 
\[
X(\un{i})^\ast = X_{i_1}^-X_{i_2}^-\cdots X_{i_\ell}^-.
\]
Note that $X(\un{i})^\ast \cdot X(\un{i})$ has degree zero and therefore can be simplified to an element of $\C[H]$.
The following lemma characterizes the zeros of these polynomials.

\begin{Lemma} \label{lem:edge-crossing}
Let $\un{i}\in\mathsf{Seq}_2$ be a binary sequence and $\la\in F$ a point in the face lattice. Then the following are equivalent:
\begin{enumerate}[{\rm (i)}]
\item $X(\un{i})^\ast \cdot  X(\un{i})$ belongs to the principal ideal $(H-\la)$ of $\C[H]$;
\item the based face lattice path $\bar\pi(\un{i},\la)$ intersects an edge in $\overline{\mathscr{L}}$.
\end{enumerate}
\end{Lemma}
\begin{proof} For $j\in\{1,2\}$, let $\si_j$ denote the automorphism of $\C[H]$ given by $\si_j(H)=H-\al_j$. Since $[H,X_i^\pm]=\pm\al_i X_i^\pm$ we have $X_i^\pm f(H) = \si_i^{\pm 1}\big(f(H)\big) X_i^\pm$ for any polynomial $f$. We have
\begin{align*}
&\phantom{\Longleftrightarrow} X^-_{i_1}X^-_{i_2}\cdots X^-_{i_\ell}\cdot X^+_{i_\ell}X^+_{i_{\ell-1}}\cdots X^+_{i_1} \in (H-\la)  \\
&\Longleftrightarrow
(\si_{i_1}\si_{i_2}\cdots\si_{i_{\ell-1}})^{-1}(X^-_{i_\ell}X^+_{i_\ell})\cdot (\si_{i_1}\si_{i_2} \cdots \si_{i_{\ell-2}})^{-1}(X^-_{i_{\ell-1}}X^+_{i_{\ell-1}}) \cdots X^-_{i_1}X^+_{i_1} \in (H-\la) \\
&\Longleftrightarrow 
\exists r\in\iv{1}{\ell}:\; X^-_{i_r}X^+_{i_r}\in \si_{i_1}\si_{i_2}\cdots \si_{i_{r-1}}\big( (H-\la)\big) \\
&\Longleftrightarrow 
\exists r\in\iv{1}{\ell}:\; P^\mathscr{L}_{i_r}(H+\al_{i_r}/2)\in \big(H-(\la+\al_{i_1}+\al_{i_2}+\cdots+\al_{i_{r-1}})\big) \\
&\Longleftrightarrow 
\exists r\in\iv{1}{\ell}:\; P^\mathscr{L}_{i_r}(\la+\al_{i_1}+\al_{i_2}+\cdots+\al_{i_{r-1}}+\al_{i_r}/2)=0 \\
&\Longleftrightarrow 
\exists r\in\iv{1}{\ell}:\; \mathscr{L}_{i_r}(\la+\al_{i_1}+\al_{i_2}+\cdots+\al_{i_{r-1}}+\al_{i_r}/2)>0 \\
&\Longleftrightarrow 
\text{The path $\bar\pi(\un{i},\la)$ intersects an edge from $\overline{\mathscr{L}}$.}
\end{align*}
\end{proof}

As a consequence we obtain the following description of the cyclic subquotients $\CC_\la$ and their gradation radicals $\CJ_\la$. 

\begin{Proposition} \label{prp:B_la-J_la-Description}
Let $\la\in F$ and $D$ be a connected component of $\T_{m,n}\setminus\overline{\mathscr{L}}$ containing $\bar\la$.
\begin{enumerate}[{\rm (a)}]
\item If $D$ is contractible, then $\CJ_\la=\bigoplus_{k\in \Z\setminus\{0\}} (\CC_\la)_{(km,kn)}$ and $\CC_\la\simeq \C\oplus \CJ_\la$.
\item If $D$ is incontractible, then $\CJ_\la=\{0\}$ and $\CC_\la\simeq\C[L,L^{-1}]$ which is a Laurent polynomial algebra in one indeterminate $L$ of degree $(m,n)$.
\end{enumerate}
\end{Proposition}
\begin{proof}
(a) Suppose $D$ is contractible. Let $k\in\Z, k\neq 0$, and let $a\in (\CC_\la)_{(km,kn)}$. After applying the involution $\ast$ if necessary, we may assume that $k>0$. Let $\un{i}\in\mathsf{Seq}_2(km,kn)$ and consider the element $X(\un{i})^\ast a$. It has degree zero and hence belongs to $C[H]$. Thus its square is 
\[(X(\un{i})^\ast  a)^2 = (X(\un{i})^\ast  a) (X(\un{i})^\ast  a)^\ast  = 
X(\un{i})^\ast  a  a^\ast  X(\un{i}) = a a^\ast  X(\un{i})^\ast X(\un{i})\]
which is equal to zero by Lemma \ref{lem:edge-crossing}. Therefore $X(\un{i})^\ast  a=0$ for all $\un{i}\in\mathsf{Seq}_2(km,kn)$. This implies that $a$ belongs to the gradation radical $\CJ_\la$ as desired.

(b) If $D$ is incontractible, then there exists a sequence $\un{i}\in\mathsf{Seq}_2(m,n)$ such that the based face lattice path $\bar\pi(\un{i},\la)$ doesn't intersect any edge in $\overline{\mathscr{L}}$. By Lemma \ref{lem:edge-crossing}, $X(\un{i})^\ast\cdot  X(\un{i})\notin (H-\la)$ in $\CC$. Thus, the image $L$ of $X(\un{i})$ in $\CC_\la$ is an invertible element of degree $(m,n)$.
\end{proof}

\begin{Corollary}\label{cor:modules}
There is a bijective correspondence between the isoclasses of simple integral weight $\CA$-modules without inner breaks, and the set of pairs $(D,\xi)$ where $D$ is a connected component of $\T_{m,n}\setminus\overline{\mathscr{L}}$ and $\xi\in\C$ with $\xi=0$ iff $D$ is contractible. Each nonzero weight space of such a module is one-dimensional, the support is given by $\{\la\in F\mid \bar\la\in D\}$.
\end{Corollary}
\begin{proof}
Follows directly from Lemma \ref{lem:edge-crossing}, Lemma \ref{lem:partition}, Proposition \ref{prp:B_la-J_la-Description} and Theorem \ref{thm:TGWA-modules-classification}.
\end{proof}

To prove Theorem \ref{thm:B}(i)-(ii), it remains to show that if $M$ is any simple integral weight $\CA$-module, then $M$ has no inner breaks. In view of Theorem \ref{thm:no-inner-breaks}, this is equivalent to showing that $\CJ_\la M_\la=0$ for any simple integral weight $\CA$-module $M$. This in turn is equivalent to showing that $\CJ_\la$ is equal to the nil radical of $\CC_\la$.

\subsection{The nilradical of $\CC_\la$}
This subsection is the technical heart of the paper. The goal is to establish that $\CJ_\la$ is a nil ideal for every $\la\in F$. That is, that every element of $\CJ_\la$ is nilpotent. First we reduce to the case of elements of degree $(m,n)$.
\begin{Lemma} \label{lem:nilpotency-reduction}
Let $\la\in F$. If $X(\un{i})$ is nilpotent in $\CC_\la$ for all $\un{i}\in \mathsf{Seq}_2(m,n)$, then $\CJ_\la$ is nil.
\end{Lemma}
\begin{proof}
If $\bar\la$ belongs to a contractible component of $\T_{m,n}\setminus\overline{\mathscr{L}}$, then $\CJ_\la=\{0\}$. So suppose $\bar\la$ belongs to an incontractible component.
Since $\CC_\la$ is a commutative $G_0$-graded algebra, it suffices to prove that any homogeneous element of nonzero degree is nilpotent. First we prove that $X(\un{j})$ is nilpotent in $\CC_\la$ for all $\un{j}\in \cup_{k=1}^\infty \mathsf{Seq}_2(km,kn)$.

Let $\un{j}\in \cup_{k=1}^\infty \mathsf{Seq}_2(km,kn)$. Then $\ell(\un{j})=km+kn$ for some $k\in\Z_{>0}$. We prove that $X(\un{j})$ is nilpotent by induction on $k$. For $k=1$ this is true by assumption. For $k>1$, $\un{j}$ is cyclically reducible (see e.g. \cite[Lem.~3.4]{HarRos2016}), meaning that $\un{j}=\un{j}'\un{i}\un{j}''$ for some $\un{i}\in\mathsf{Seq}_2(m,n)$ and some sequences $\un{j}', \un{j}''\in\mathsf{Seq}_2$ such that $\un{j}'\un{j}''\in \cup_{k=1}^\infty \mathsf{Seq}_2(km,kn)$.
Put $a=X(\un{j}')$, $b=X(\un{i})$, $c=X(\un{j}'')\in A$.
Since the cyclic subalgebra $\CC$ is commutative and $b=X(\un{i}), ac=X(\un{j}'\un{j}''), ca=X(\un{j}''\un{j}')\in \CC$ we have,
\begin{equation}\label{eq:nilpotency}
(abc)^N=a(bca)^{N-1}bc=a(ca)^{N-1}b^nc=(ac)^{N-1}ab^nc
\end{equation}
in $\CC$. Since $\ell(\un{j}'\un{j}'')=(k-1)(m+n)$, the induction hypothesis implies that $ac$ is nilpotent in $\CC_\la$. By \eqref{eq:nilpotency}, this implies that $abc$ is nilpotent in $\CC_\la$.

By applying the involution we also get the same result for $X(\un{j})^\ast$. Finally, let $a$ be an arbitrary homogeneous element of $\CC$ of degree $(km,kn)$, $k\neq 0$. Then $a$ is a $\C[H]$-linear combination of elements of the form $X(\un{j})$ and $X(\un{j})^\ast$, where $\un{j}\in \mathsf{Seq}_2(km,kn)$, and thus $a$ is nilpotent.
\end{proof}

In order to proceed we need to introduce some combinatorial quantities.
\begin{Definition}[Order of a vertex]
The \emph{order} of a vertex $v\in V$ with respect to $\mathscr{L}$ is
\begin{align*}
\ord(v)=\ord_\mathscr{L}(v)&=\mathscr{L}_2(v+\al_1/2)-\mathscr{L}_2(v-\al_1/2)
 =(\#\text{edges right of $v$}) - (\#\text{edges left of $v$})\\
&=\mathscr{L}_1(v-\al_2/2)-\mathscr{L}_1(v+\al_2/2)
 =(\#\text{edges below $v$}) - (\#\text{edges above $v$})
\end{align*}
The fourth equality is due to the ice rule \eqref{eq:ice-rule-alpha}.
A vertex $v\in V$ is called a \emph{corner} if its order is nonzero. 
See Figures \ref{fig:order-def} and \ref{fig:53-example-with-orders}.
\end{Definition}
Note that the definition of $\ord(v)$ breaks the $(1\leftrightarrow 2)$-symmetry. The dual way to define it would equal $-\ord(v)$.

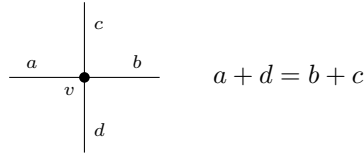
\begin{figure}
\centering 
\[
\begin{tikzpicture}[baseline={([yshift=-.5ex]current bounding box.center)}]
\draw (-1, 0)--( 1, 0);
\draw ( 0,-1)--( 0, 1);
\fill (0,0) circle (2pt);
\node[font=\scriptsize, anchor=north east] at (0,0) {$v$};
\node[font=\scriptsize, anchor=south] at (-.7,0) {$a$};
\node[font=\scriptsize, anchor=south] at ( .7,0) {$b$};
\node[font=\scriptsize, anchor=west] at (0, .7) {$c$};
\node[font=\scriptsize, anchor=west] at (0,-.7) {$d$};
\end{tikzpicture}
\qquad a+d=b+c 
\]
\caption{The order of $v$ is defined as $\ord(v)=b-a=d-c$.}
\label{fig:order-def}
\end{figure}

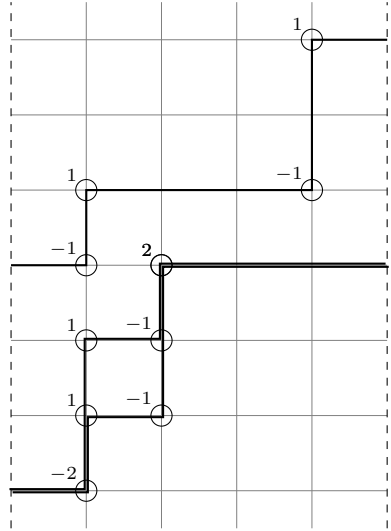
\begin{figure}
\centering
\[
\begin{tikzpicture}
% Help lines
\foreach \x in {0,...,3} { \draw[help lines] (\x,-.5)--(\x,6.5); }
\foreach \y in {0,...,6} { \draw[help lines] (-1,\y)--(4,\y); }
% Seams
\draw[dashed] (-1,-.5)--(-1,6.5);
\draw[dashed] (4,-.5)--(4,6.5);
% Vertex paths
\draw[thick,xshift=.6pt,yshift=-.6pt] (-1,0)--(0,0)--(0,1)--(1,1)--(1,2)--(1,3)--(2,3)--(3,3)--(4,3);
\draw[thick,xshift=-.6pt,yshift=.6pt] (-1,0)--(0,0)--(0,1)--(0,2)--(1,2)--(1,3)--(2,3)--(3,3)--(4,3);
\draw[thick] (-1,3)--(0,3)--(0,4)--(1,4)--(2,4)--(3,4)--(3,5)--(3,6)--(4,6);
% Corners
\draw (0,0) circle (4pt) node[above left] {${\scriptstyle -2}$}; 
\draw (0,1) circle (4pt) node[above left] {${\scriptstyle 1}$};
\draw (0,2) circle (4pt) node[above left] {${\scriptstyle 1}$};
\draw (0,3) circle (4pt) node[above left] {${\scriptstyle -1}$};
\draw (0,4) circle (4pt) node[above left] {${\scriptstyle 1}$};
\draw (1,1) circle (4pt) node[above left] {${\scriptstyle -1}$};
\draw (1,2) circle (4pt) node[above left] {${\scriptstyle -1}$};
\draw (1,3) circle (4pt) node[above left] {${\scriptstyle 2}$};
\draw (1,3) circle (4pt) node[above left] {${\scriptstyle 2}$};
\draw (3,4) circle (4pt) node[above left] {${\scriptstyle -1}$};
\draw (3,6) circle (4pt) node[above left] {${\scriptstyle 1}$};
\end{tikzpicture}
\]
\caption{A corner is a vertex of nonzero order. Here a $(5,3)$-periodic configuration $\mathscr{L}$ is shown where the corners have been circled and their respective order indicated.}
\label{fig:53-example-with-orders}
\end{figure}

\begin{Definition}[Order of a face path]
Let $\un{i}\in\mathsf{Seq}_2$ and $\la\in F$ and consider the corresponding based face path
 $\pi(\un{i},\la)$ in the face lattice $\Z^2$ of the plane. We define the \emph{order} of such a path to be
\begin{equation}
\ord(\un{i},\la)= \sum \ord(v)
\end{equation}
where we sum over all vertices $v\in V$ that lie directly above some horizontal edge or directly to the left of some vertical edge of the path $\pi(\un{i},\la)$. 
See Figure \ref{fig:53-example-with-face-path}.
\end{Definition}

Note that in this situation we think of the lattice path as a path in the plane rather than on the cylinder. This becomes relevant for situations where the path would otherwise wrap around the cylinder several times.

\begin{figure}
\centering
\[
\begin{tikzpicture}
% Help lines
\foreach \x in {0,...,3} { \draw[help lines] (\x,-.5)--(\x,6.5); }
\foreach \y in {0,...,6} { \draw[help lines] (-1,\y)--(4,\y); }
% Seams
\draw[dashed] (-1,-.5)--(-1,6.5);
\draw[dashed] (4,-.5)--(4,6.5);
% Vertex paths
\draw[thick,xshift=.6pt,yshift=-.6pt] (-1,0)--(0,0)--(0,1)--(1,1)--(1,2)--(1,3)--(2,3)--(3,3)--(4,3);
\draw[thick,xshift=-.6pt,yshift=.6pt] (-1,0)--(0,0)--(0,1)--(0,2)--(1,2)--(1,3)--(2,3)--(3,3)--(4,3);
\draw[thick] (-1,3)--(0,3)--(0,4)--(1,4)--(2,4)--(3,4)--(3,5)--(3,6)--(4,6);
% Face path
\draw[thick,Blue] (-1,.5)--(-.5,.5)--(-.5,1.5)--(.5,1.5)--(1.5,1.5)--(1.5,2.5)--(2.5,2.5)--(2.5,3.5)--(3.5,3.5)--(4,3.5);
% Face point
\draw (-.5cm-2pt,1.5cm-2pt)--(-.5cm+2pt,1.5cm+2pt);
\draw (-.5cm+2pt,1.5cm-2pt)--(-.5cm-2pt,1.5cm+2pt);
\draw (-.5,1.5) node[above] {${\scriptstyle \bar\lambda}$};
\draw (.5cm-2pt,1.5cm-2pt)--(.5cm+2pt,1.5cm+2pt);
\draw (.5cm+2pt,1.5cm-2pt)--(.5cm-2pt,1.5cm+2pt);
\draw (.5,1.5) node[above] {${\scriptstyle \bar\mu}$};
% Corners
\draw (0,0) circle (4pt) node[above left] {${\scriptstyle -2}$}; 
\draw (0,1) circle (4pt) node[above left] {${\scriptstyle 1}$};
\draw (0,2) circle (4pt) node[above left] {${\scriptstyle 1}$};
\draw (0,3) circle (4pt) node[above left] {${\scriptstyle -1}$};
\draw (0,4) circle (4pt) node[above left] {${\scriptstyle 1}$};
\draw (1,1) circle (4pt) node[above left] {${\scriptstyle -1}$};
\draw (1,2) circle (4pt) node[above left] {${\scriptstyle -1}$};
\draw (1,3) circle (4pt) node[above left] {${\scriptstyle 2}$};
\draw (1,3) circle (4pt) node[above left] {${\scriptstyle 2}$};
\draw (3,4) circle (4pt) node[above left] {${\scriptstyle -1}$};
\draw (3,6) circle (4pt) node[above left] {${\scriptstyle 1}$};
\end{tikzpicture}
\]
\caption{A $(5,3)$-periodic configuration $\mathscr{L}$. The order of the blue path in the face lattice is the sum of all orders of corners above it, which in this case equals $2$. This is also equal to the number of horizontal (respectively vertical) edges that the face path crosses, taken with multiplicity.}
\label{fig:53-example-with-face-path}
\end{figure}

Notice that the face lattice path $\pi(\un{i},\la)$ in Figure \ref{fig:53-example-with-face-path} crosses two vertical and two horizontal edges, and that $\ord(\un{i},\la)=2$. We show that this is a general phenomenon. This gives an equivalent way of defining the order of a path.

\begin{Lemma}\label{lem:order}
The order of a face path is equal to $\max(v,h)$ where $v$ (respectively $h$) is the number of vertical (respectively horizontal) edges in $\mathscr{L}$ that it crosses.
\end{Lemma}
\begin{proof}
By additivity, it suffices to prove the statement when $\mathscr{L}$ consists of a single vertex path of period $(m,n)$. Similarly we may assume that the face path has length $1$, i.e. consists of a single horizontal or vertical step. By symmetry we can assume it is a horizontal step. 
The step crosses either above or below $\mathscr{L}$, or between two corners, see Figure \ref{fig:path-crossing}. If it crosses above, the order is zero. If it crosses below the order is $-1+1=0$ since both corners are counted. Finally, if it crosses between the corners the order is $1$ since the top corner has order $1$ and the bottom corner $-1$. 
\end{proof}

\begin{figure}
\centering 
\begin{tikzpicture}
\draw (-1,0)--(0,0)--(0,1)--(1,1);
\fill (-1,0) circle (2pt);
\fill (0,0) circle (2pt) node[above left] {${\scriptstyle -1}$};
\fill (0,1) circle (2pt) node[above left] {${\scriptstyle 1}$};
\fill (1,1) circle (2pt);
\draw[Blue,thick] (-.5,1.5)--(.5,1.5);
\draw (-.5cm-2pt,1.5cm-2pt)--(-.5cm+2pt,1.5cm+2pt);
\draw (-.5cm+2pt,1.5cm-2pt)--(-.5cm-2pt,1.5cm+2pt);
\draw (.5cm-2pt,1.5cm-2pt)--(.5cm+2pt,1.5cm+2pt);
\draw (.5cm+2pt,1.5cm-2pt)--(.5cm-2pt,1.5cm+2pt);
\end{tikzpicture}
\qquad 
\begin{tikzpicture}
\draw (-1,0)--(0,0)--(0,1)--(1,1);
\fill (-1,0) circle (2pt);
\fill (0,0) circle (2pt) node[above left] {${\scriptstyle -1}$};
\fill (0,1) circle (2pt) node[above left] {${\scriptstyle 1}$};
\fill (1,1) circle (2pt);
\draw[Blue,thick] (-.5,.5)--(.5,.5);
\draw (-.5cm-2pt,.5cm-2pt)--(-.5cm+2pt,.5cm+2pt);
\draw (-.5cm+2pt,.5cm-2pt)--(-.5cm-2pt,.5cm+2pt);
\draw (.5cm-2pt,.5cm-2pt)--(.5cm+2pt,.5cm+2pt);
\draw (.5cm+2pt,.5cm-2pt)--(.5cm-2pt,.5cm+2pt);
\end{tikzpicture}
\qquad 
\begin{tikzpicture}
\draw (-1,0)--(0,0)--(0,1)--(1,1);
\fill (-1,0) circle (2pt);
\fill (0,0) circle (2pt) node[above left] {${\scriptstyle -1}$};
\fill (0,1) circle (2pt) node[above left] {${\scriptstyle 1}$};
\fill (1,1) circle (2pt);
\draw[Blue,thick] (-.5,-.5)--(.5,-.5);
\draw (-.5cm-2pt,-.5cm-2pt)--(-.5cm+2pt,-.5cm+2pt);
\draw (-.5cm+2pt,-.5cm-2pt)--(-.5cm-2pt,-.5cm+2pt);
\draw (.5cm-2pt,-.5cm-2pt)--(.5cm+2pt,-.5cm+2pt);
\draw (.5cm+2pt,-.5cm-2pt)--(.5cm-2pt,-.5cm+2pt);
\end{tikzpicture}
\caption{Possible ways a horizontal face step can cross a vertex path $\mathscr{L}$.}
\label{fig:path-crossing}
\end{figure}
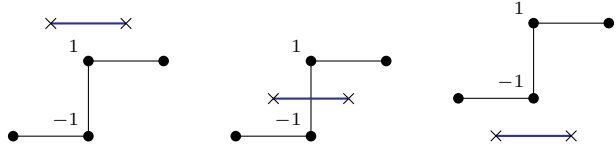

\begin{Corollary}
$\ord(\un{i},\la)\ge 0$ for all $\la\in F$ and all $\un{i}\in\mathsf{Seq}_2$.
\end{Corollary}

\begin{Corollary}\label{cor:order-positivity}
Let $\la\in F$ and $\un{i}\in\mathsf{Seq}_2(m,n)$. If the based face path $\bar\pi(\un{i},\la)$ passes through a contractible connected component $D$ of $\T_{m,n}\setminus\overline{\mathscr{L}}$, then $\ord(\un{i},\la)>0$.
\end{Corollary}
\begin{proof}
$\bar\pi(\un{i},\la)$ has to cross at least one edge in $\overline{\mathscr{L}}$, otherwise it is an incontractible loop in $D$.
\end{proof}

Next we derive useful exchange relations. Put $\CA_\mathrm{loc}:=\CA\otimes_{\C[H]}\C(H)$.

\begin{Lemma}[Exchange relation]
In $\CA_\mathrm{loc}$ we have
\begin{equation}\label{eq:flip-identity}
X_2^+X_1^+ = X_1^+X_2^+ \tilde{R}(H),
\end{equation}
where, putting $\rho=(\al_1+\al_2)/2$,
\begin{equation}
\tilde{R}(H) = \prod_{\la\in F} (H-\la)^{\ord(\la+\rho)}.
\end{equation}
\end{Lemma}
\begin{proof}
 By the TGWA exchange relation \eqref{eq:TGWA-exchange-relation} or using the defining relations \eqref{eq:Aalbepq-rels} of $\CA$ we have
\[X_2^+X_1^+ P_1^\mathscr{L}(H+\rho+\al_2/2) = X_1^+X_2^+ P_1^\mathscr{L}(H+\rho-\al_2/2),\]
where $\rho=(\al_1+\al_2)/2$. Now $v=\la+\rho\in V$ is a zero of $P_1^\mathscr{L}(u-\al_2/2)$ of multiplicity $d$ iff $v-\al_2/2$ is a zero of $P_1^\mathscr{L}(u)$ of multiplicity $d$, i.e. iff $e=v-\al_2/2$ is the midpoint of a vertical edge with label $\mathscr{L}_1(e)=d$. Similarly $v\in V$ is a zero of $P_1^\mathscr{L}(u+\al_2/2)$ of multiplicity $c$ iff $\mathscr{L}_1(v+\al_2/2)=c$. Thus the factor $(u-v)$ appears with multiplicity $d-c=\ord(v)$ in $\tilde{R}(u-\rho)=P_1^\mathscr{L}(u-\al_2/2)/P_1^\mathscr{L}(u+\al_2/2)$. Substituting $u=H+\rho$, the claim is proved.
\end{proof}

\begin{Lemma}[Generalized exchange relation]\label{lem:GER}
For any pair of non-negative integers $(r,s)$ and sequences $\un{i}, \un{j}\in\mathsf{Seq}_2(r,s)$, the following identity holds in $\CA_{\mathrm{loc}}$:
\begin{equation}\label{eq:GER}
X(\un{i})=X(\un{j}) \cdot \prod_{\la\in F} (H-\la)^{\ord(\un{i},\la)-\ord(\un{j},\la)}.
\end{equation}
\end{Lemma}
\begin{proof}
Put $X_i=X_i^+$ and $N=r+s$.
Multiplying both sides of the exchange relation \eqref{eq:flip-identity} by products of generators $X_i$ and
using that $HX_i = X_i(H+\al_i)$ we get
\begin{equation}\label{eq:GER-proof-1}
X_{i_N}\cdots X_{i_{k+2}}\textcolor{blue}{X_2 X_1} X_{i_{k-1}} \cdots X_{i_1} = 
X_{i_N}\cdots X_{i_{k+2}}\textcolor{red}{X_1X_2}X_{i_{k-1}} \cdots X_{i_1} \cdot \prod_{\la\in F} (H-\la)^{\ord(\la+\ga+\rho)}
\end{equation}
where $\ga=\al_{i_1}+\al_{i_2}+\cdots+\al_{i_{k-1}}$. Note that the vertex $v=\la+\ga+\rho$ is precisely the vertex between the two based face lattice paths, see Figure \ref{fig:GER-proof-step1}.
Iterating this, 
for any two sequences $\un{i},\un{j}\in\mathsf{Seq}_2(r,s)$ such that the based face lattice path $\bar\pi(\un{i},\la)$ is below $\bar\pi(\un{j},\la)$, we have
\begin{equation}\label{eq:GER-proof-2}
\textcolor{blue}{X(\un{i})}=\textcolor{red}{X(\un{j})} \cdot \prod_{\la\in F} (H-\la)^{\sum \ord(v)}
\end{equation}
where in the exponent we sum over all vertices $v$ that are between $\bar\pi(\un{i},\la)$ and $\bar\pi(\un{j},\la)$, see Figure \ref{fig:GER-proof-step2}. Note that $\sum \ord(v)= \ord(\un{i},\la)-\ord(\un{j},\la)$, which proves \eqref{eq:GER} in this case. The general case, in which the paths may cross at one or more points of the face lattice, can be obtained by splitting them into pieces.
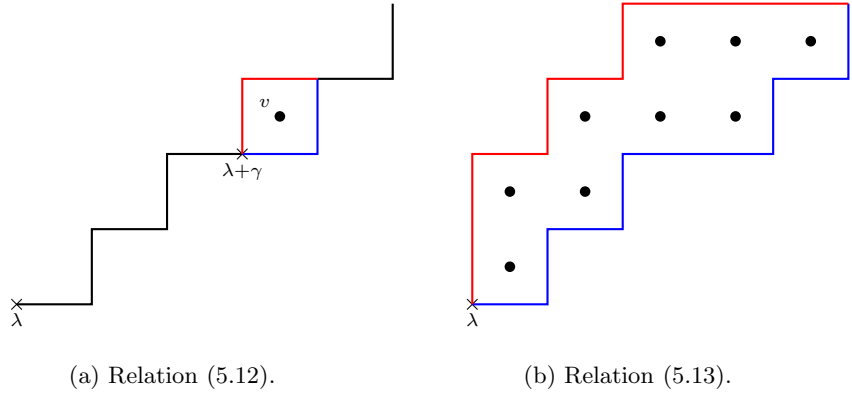
\begin{figure}
\centering
\begin{subfigure}[b]{.3\textwidth}
\[
\begin{tikzpicture}
\draw (0cm-2pt,0cm-2pt)--(0cm+2pt,0cm+2pt);
\draw (0cm+2pt,0cm-2pt)--(0cm-2pt,0cm+2pt);
\draw (0,0) node[below] {${\scriptstyle \lambda}$};

\draw (3cm-2pt,2cm-2pt)--(3cm+2pt,2cm+2pt);
\draw (3cm+2pt,2cm-2pt)--(3cm-2pt,2cm+2pt);
\draw (3,2) node[below] {${\scriptstyle \lambda+\gamma}$};

\draw[-,thick] (0,0)--(1,0)--(1,1)--(2,1)--(2,2)--(3,2);
\draw[-,thick] (4,3)--(5,3)--(5,4);
\draw[-,thick,blue] (3,2)--(4,2)--(4,3);
\draw[-,thick,red] (3,2)--(3,3)--(4,3);

\fill (3.5,2.5) circle (2pt) node[above left] {${\scriptstyle v}$};
\end{tikzpicture}
\]
\caption{Relation \eqref{eq:GER-proof-1}.}
\label{fig:GER-proof-step1}
\end{subfigure}
\qquad\qquad 
\begin{subfigure}[b]{.3\textwidth}
\[
\begin{tikzpicture}
\draw (0cm-2pt,0cm-2pt)--(0cm+2pt,0cm+2pt);
\draw (0cm+2pt,0cm-2pt)--(0cm-2pt,0cm+2pt);
\draw (0,0) node[below] {${\scriptstyle \lambda}$};

\draw[-,thick,blue] (0,0)--(1,0)--(1,1)--(2,1)--(2,2)--(4,2)--(4,3)--(5,3)--(5,4);
\draw[-,thick,red]  (0,0)--(0,2)--(1,2)--(1,3)--(2,3)--(2,4)--(5,4);

\fill (0.5,0.5) circle (2pt) node[above left] {};
\fill (0.5,1.5) circle (2pt) node[above left] {};
\fill (1.5,1.5) circle (2pt) node[above left] {};
\fill (1.5,2.5) circle (2pt) node[above left] {};
\fill (2.5,2.5) circle (2pt) node[above left] {}; 
\fill (2.5,3.5) circle (2pt) node[above left] {}; 
\fill (3.5,2.5) circle (2pt) node[above left] {};
\fill (3.5,3.5) circle (2pt) node[above left] {};
\fill (4.5,3.5) circle (2pt) node[above left] {};
\end{tikzpicture}
\]
\caption{Relation \eqref{eq:GER-proof-2}.}
\label{fig:GER-proof-step2}
\end{subfigure}
\caption{Special cases of the generalized exchange relation.}
\label{fig:GER-proof-step}
\end{figure}
\end{proof}

We note the following corollary to Lemma \ref{lem:GER}.
\begin{Corollary}\label{cor:quotient-basis}
Fix $\la\in F$ and $(a,b)\in\N^2$. Choose $\un{i}\in\mathsf{Seq}_2(a,b)$ such that $\ord(\un{i},\la)$ is minimal. Let $J_\la$ be the left ideal in $\CA$ generated by $H-\la$. Then $(\CA/J_\la)_{(a,b)}$ is $1$-dimensional with basis the image of $X(\un{i})$.
\end{Corollary}
\begin{proof}
Let $\bar X(\un{i})$ be the image of $X(\un{i})$ in the quotient $\CA/J_\la$.
First we show that $\bar X(\un{i})$ spans.
For any other $\un{j}\in\mathsf{Seq}_2(a,b)$ we have $\ord(\un{j},\la)-\ord(\un{i},\la)\ge 0$.
So $X(\un{j})\in \C X(\un{i}) + J_\la$ in $\CA$ by Lemma \ref{lem:GER}. Since the monomials $X(\un{j})$ generate $\CA_{(a,b)}$ as a right $\C[H]$-module, their images span $(\CA/J_\la)_{(a,b)}$, which proves the claim.

Next we show that $\bar X(\un{i})$ is nonzero. Put $g=(a,b)$. Assume that $(J_\la)_g=\CA_g$. Fix $b\in A_{-g}\setminus\{0\}$. Define $\varphi:\CA_g\to R$ by $\varphi(a)=ba$. Then $\varphi(ar)=\varphi(a)r$ for all $r\in R$, hence the image of $\varphi$ is an ideal of $R$. This ideal is nonzero since the gradation form is non-degenerate. Since $R$ is a PID, there is a single generator $f\in R$ for this ideal. Pick $F\in \CA_g$ with $\varphi(F)=f$. By the assumption, $F=F_1\cdot (H-\la)$ for some $F_1\in \CA_g$. Hence $f=\varphi(F)=\varphi(F_1\cdot(H-\la))=\varphi(F_1)\cdot(H-\la_1)=fg\cdot(H-\la)$ for some $g\in R$, since $\varphi(F_1)\in (f)$. Dividing by $f$ we obtain $1=g\cdot(H-\la)$ which is a contradiction.
\end{proof}

\begin{Lemma}\label{lem:flip-lemma}
Put $X_i=X_i^+$.
Let $\un{i}\in \mathsf{Seq}_2(m,n)$ and let $\la\in F$.
Then there exists $a_0\in\N$ such that for all integers $a\ge a_0$ there is a rational function $f_a(u)\in\C(u)$, which is regular and nonzero at $u=\la$, such that
\begin{equation}
X_1^{am}X_2^{an} X(\un{i}) = X_1^{(a+1)m}X_2^{(a+1)n} (H-\la)^{\ord(\un{i},\la)} f_a(H).
\end{equation}
\end{Lemma}
\begin{proof}
Choose $a_0\in\N$ large enough to ensure that there are no corners above the path $\pi(\un{i},\la+a_0\al_2)$.
Then use the generalized exchange relation \eqref{eq:GER}.
See Figure \ref{fig:flip-lemma}.
\end{proof}

\begin{figure}
\centering
\[
\begin{tikzpicture}[xscale=.5,yscale=.5]
% Seams
\draw[dashed] (-.5,-.7)--(-.5,12.7);
\draw[dashed] (4.5,-.7)--(4.5,12.7);
\draw[dashed] (9.5,-.7)--(9.5,12.7);
\draw[dashed] (14.5,-.7)--(14.5,12.7);
% Face points
\draw (0cm-4pt,0cm-4pt)--(0cm+4pt,0cm+4pt);
\draw (0cm+4pt,0cm-4pt)--(0cm-4pt,0cm+4pt);
\draw (5cm-4pt,4cm-4pt)--(5cm+4pt,4cm+4pt);
\draw (5cm+4pt,4cm-4pt)--(5cm-4pt,4cm+4pt);
\draw (15cm-4pt,12cm-1.5pt-4pt)--(15cm+4pt,12cm-1.5pt+4pt);
\draw (15cm+4pt,12cm-1.5pt-4pt)--(15cm-4pt,12cm-1.5pt+4pt);
% Face paths
\draw[thick,Blue] (0,0)--(2,0)--(2,2)--(3,2)--(3,4)--(5,4);
\draw[thick,OliveGreen] (5,4)--(5,12cm-3pt)--(15,12cm-3pt);
\draw[thick,Red] (0,0)--(0,12)--(15,12);
% Corners
\draw (3.5,1.5) circle (4pt) node[above left] {};
\draw (1.5,2.5) circle (4pt) node[above left] {};
\draw (0.5,3.5) circle (4pt) node[above left] {};
\draw (2.5,6.5) circle (4pt) node[above left] {};
\draw (.5,8.5) circle (4pt) node[above left] {};
\draw (1.5,10.5) circle (4pt) node[above left] {};
\end{tikzpicture}
\]
\caption{$\textcolor{OliveGreen}{X_1^{am}X_2^{an}} \textcolor{blue}{X(\un{i})} = \textcolor{red}{X_1^{(a+1)m}X_2^{(a+1)n} } (H-\la)^{\ord(\un{i},\la)} f_a(H)$.}
\label{fig:flip-lemma}
\end{figure}

\begin{Lemma} \label{lem:nil-thm}
Put $X_i=X_i^+$.
Let $\un{i}\in \mathsf{Seq}_2(m,n)$ and $\la\in F$. Then there exist $a_0\in\N$ and $b_0\in\Z$ such that
for all $a\ge a_0$ there is a rational function $g_a(u)\in \C(u)$, which is regular and nonzero at $u=\la$, such that
\begin{equation}
X(\un{i})^a = X_1^{am}X_2^{an} (H-\la)^{b_0+a\ord(\un{i},\la)} g_a(H).
\end{equation}
\end{Lemma}
\begin{proof}
Choose $a_0$ to be the integer from Lemma \ref{lem:flip-lemma}. We use induction on $a$ to prove the claim.
For $a=a_0$, the existence of $b_0$ and $g_a$ follow from the generalized exchange relation \eqref{eq:GER}.
Suppose the claim is true for some integer $a\ge a_0$.
Then since $X(\un{i})$ commutes with $H$,
\[X(\un{i})^{a+1} =
X_1^{am}X_2^{an} X(\un{i}) (H-\la)^{b_0+a\ord(\un{i},\la)} g_a(H).\]
By Lemma \ref{lem:flip-lemma} the right hand side equals
\[
X_1^{(a+1)m}X_2^{(a+1)n} (H-\la)^{b_0+(a+1)\ord(\un{i},\la)} f_a(H) g_a(H).
\]
Taking $g_{a+1}(u)=f_a(u)g_a(u)$, this finishes the proof of the induction step.
\end{proof}

We are now ready to prove that the ideal $\CJ_\la$ is nil.

\begin{Theorem}\label{thm:radicals}
For any $\la\in F$, the gradation radical $\CJ_\la$ of $\CC_\la$ is equal to the nilradical of $\CC_\la$. In particular, $\CJ_\la$ is nil.
\end{Theorem}
\begin{proof}
If $\bar\la$ belongs to an incontractible connected component of $\T_{m,n}\setminus\overline{\mathscr{L}}$, then $\CC_\la$ is a domain and $\CJ_\la =\{0\}$ by Proposition \ref{prp:B_la-J_la-Description}.

Suppose $\bar\la$ belongs to a contractible connected component. The homogeneous term of degree $(0,0)$ of any nilpotent element $a$ of $\CC_\la$ must be zero, hence $a\in \CJ_\la$. For the converse, by Lemma \ref{lem:nilpotency-reduction} it is enough to prove that $X(\un{i})$ is nilpotent in $\CC_\la$ for all $\un{i}\in\mathsf{Seq}_2(m,n)$.
By Corollary \ref{cor:order-positivity}, $\ord(\un{i},\la)>0$, hence by Lemma \ref{lem:nil-thm}, $X(\un{i})^a\in \CC\cdot (H-\la)$ for $a\gg 0$.
\end{proof}

\begin{Corollary}
$\CJ_\la M_\la=0$ for any simple integral weight module $M$ and any $\la\in F$.
\end{Corollary}
\begin{proof}
Since $M_\la$ is a simple finite-dimensional module over $\CC_\la$, and every element of $\CJ_\la$ is nilpotent, it follows that $\CJ_\la M_\la=0$.
\end{proof}

\begin{Corollary}\label{cor:no-inner-breaks}
Let $M$ be a simple weight $\CA$-module. Then $M$ has no inner breaks.
\end{Corollary}

In view of Corollary \ref{cor:modules}, this completes the proof of Theorem \ref{thm:B}(i)-(ii). It is an easy exercise to prove that any finite-dimensional simple $\CA$-module is a weight module, proving part (iv). Part (iii) follows from  Proposition \ref{prp:centralizing-element}(c) to be proved in the next section.

\begin{Example}
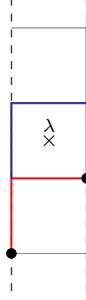
\begin{figure}
\centering
\begin{tikzpicture}[xscale=1,yscale=1]
% Help lines
\draw[help lines] (0,0)--(1,0);
\draw[help lines] (0,1)--(1,1);
\draw[help lines] (0,2)--(1,2);
\draw[help lines] (0,3)--(1,3);
% Seams
\draw[dashed] (0,-.5)--(0,3.5);
\draw[dashed] (1,-.5)--(1,3.5);
% First path
\draw[-,thick,Red] (0,0)--(0,1)--(1,1)--(1,2);
% Second path
\draw[-,thick,Blue] (0,1)--(0,2)--(1,2)--(1,3);
% Vertices
\fill (0,0) circle (2pt);
\fill (1,1) circle (2pt);
% Weight spaces
\draw (.5cm-2pt,1.5cm-2pt)--(.5cm+2pt,1.5cm+2pt);
\draw (.5cm+2pt,1.5cm-2pt)--(.5cm-2pt,1.5cm+2pt);
\draw (.5,1.5) node[font=\scriptsize, above] {$\la$};
\end{tikzpicture}
\caption{A $(1,1)$-periodic higher spin vertex configuration.} \label{fig:11-config}
\end{figure}
Let $(m,n)=(1,1)$ and let $\mathscr{L}$ be as in Figure \ref{fig:11-config}, where $\la=0$.
With $(\al_1,\al_2)=(-1,1)$ we have $P_1^\mathscr{L}(u)=P_2^\mathscr{L}(u)=(u-1/2)(u+1/2)$.
Let $\la=0$. Then $\bar\la$ belongs to a simply connected component.
Using the exchange relation \eqref{eq:flip-identity} one can check directly that $X_2X_1X_2\in \CA(H-\la)$ and $X_1X_2X_1\in \CA(H-\la)$. This implies that $X_1X_2$ and $X_2X_1$ are nilpotent in $\CB_\la$. Since $(12)$ and $(21)$ are the only elements of $\mathsf{Seq}_2(1,1)$, it follows by Lemma \ref{lem:nilpotency-reduction} that every element of $J_\la$ is nilpotent. Thus $J_\la$ has to act as zero on any simple weight module with $\la$ in its support.
\end{Example}

\begin{Example}
Let $(m,n)=(1,1)$, $(\al_1,\al_2)=(-1,1)$. Put $\la=0$ and consider the higher spin vertex configuration $\mathscr{L}$ in Figure \ref{fig:order7}. 
Let $p_i(u)=P_i^\mathscr{L}(u)$. In this case $p_1(u)$ and $p_2(u)$ coincide (which can only happen when $(m,n)=(1,1)$) and are given by
\[
p_1(u)=p_2(u)= \textcolor{Blue}{(u-1/2)} \textcolor{OliveGreen}{(u+1/2)^2(u-3/2)^2} \textcolor{Purple}{(u+3/2)^3(u-5/2)^3}.
\]
In $\CA(\mathscr{L})$ we have by \eqref{eq:Aalbepq-rels}
\[HX_1 = X_1 (H-1)\qquad HX_2 = X_2(H+1),\]
where we put $X_i=X_i^+$.
The exchange relation \eqref{eq:flip-identity} in this case is
\begin{equation}
X_2X_1=X_1X_2\cdot \frac{(H+1)H(H-3)^3}{(H+2)^3(H-1)(H-2)}
\end{equation}
Thus for all positive integers $k$ there is a rational function $f_k$ such that
\[
(X_1X_2)^k = X_1^k X_2^k f_k(H).
\]
One checks that for $k\le 5$,  $f_k$ has a pole at $H=0$, while for $k=6$, $f_k$ is regular but nonzero at $H=0$.
However, for $k=7$, $f_k(H)$ has a zero at $H=0$ of multiplicity $1$.
This proves that $(X_1X_2)^7=0$ in the quotient algebra $\CC_\la=\CC/(H-\la)$.
Moreover, by Corollary \ref{cor:quotient-basis}, $(X_1X_2)^6\neq 0$ in $\CC_\la$.

\begin{figure}
\centering 
\[
\begin{tikzpicture}[xscale=.5,yscale=.5]
\foreach \t in {0,2,...,12}
{
\draw[-,xshift= 3pt,yshift=-3pt,Purple] (0+\t,4+\t)--(2+\t,4+\t)--(2+\t,6+\t);
\draw[-,xshift= 0pt,yshift= 0pt,Purple] (0+\t,4+\t)--(2+\t,4+\t)--(2+\t,6+\t);
\draw[-,xshift=-3pt,yshift= 3pt,Purple] (0+\t,4+\t)--(2+\t,4+\t)--(2+\t,6+\t);

\draw[-,xshift= 1.5pt,yshift=-1.5pt,OliveGreen] (2+\t,2+\t)--(2+\t,4+\t)--(4+\t,4+\t);
\draw[-,xshift=-1.5pt,yshift= 1.5pt,OliveGreen] (2+\t,2+\t)--(2+\t,4+\t)--(4+\t,4+\t);

\draw[-,Blue] (2+\t,2+\t)--(4+\t,2+\t)--(4+\t,4+\t);

\draw[-,xshift= 1.5pt,yshift=-1.5pt,OliveGreen] (4+\t,0+\t)--(4+\t,2+\t)--(6+\t,2+\t);
\draw[-,xshift=-1.5pt,yshift= 1.5pt,OliveGreen] (4+\t,0+\t)--(4+\t,2+\t)--(6+\t,2+\t);

\draw[-,xshift= 3pt,yshift=-3pt,Purple] (4+\t,0+\t)--(6+\t,0+\t)--(6+\t,2+\t);
\draw[-,xshift= 0pt,yshift= 0pt,Purple] (4+\t,0+\t)--(6+\t,0+\t)--(6+\t,2+\t);
\draw[-,xshift=-3pt,yshift= 3pt,Purple] (4+\t,0+\t)--(6+\t,0+\t)--(6+\t,2+\t);

\draw (0+\t,4+\t) circle (4pt) node[above left]{${\scriptscriptstyle 3}$};
\draw (2+\t,4+\t) circle (4pt) node[above left]{${\scriptscriptstyle -1}$};
\draw (4+\t,4+\t) circle (4pt) node[above left]{${\scriptscriptstyle -1}$};

\draw (4+\t,2+\t) circle (4pt) node[above left]{${\scriptscriptstyle 1}$};
\draw (6+\t,2+\t) circle (4pt) node[above left]{${\scriptscriptstyle 1}$};
\draw (6+\t,0+\t) circle (4pt) node[above left]{${\scriptscriptstyle -3}$};

}

\foreach \t in {0,2,...,12}
{
\draw[-,Orange] (3+\t,1+\t)--(3+\t,3+\t)--(5+\t,3+\t);
}

\foreach \t in {0,2,...,14}
{
\draw (\t cm+3cm-3pt,\t cm+1cm-3pt)--(\t cm+3cm+3pt,\t cm+1cm+3pt);
\draw (\t cm+3cm+3pt,\t cm+1cm-3pt)--(\t cm+3cm-3pt,\t cm+1cm+3pt);
\draw (\t+3,\t+1) node[above left] {${\scriptscriptstyle \lambda}$};
}

\draw[dashed,Orange] (3,3)--(3,15)--(15,15);
\end{tikzpicture}
\]
\caption{Part of a $(1,1)$-periodic higher spin vertex configuration. Since the face path starting at $\lambda$ and going up-right six times, minimizes the number of edge-crossings (six vertical/horizontal), we have $(X_1^+X_2^+)^6\neq 0$ in $\CJ_\la$. However with seven up-right steps (giving seven vertical/horizontal edge-crossings) one can do better: seven steps up followed by seven steps right gives only six edge-crossings. Therefore in $\CJ_\la$ we have $(X_1^+X_2^+)^7=0$.}
\label{fig:order7}
\end{figure}
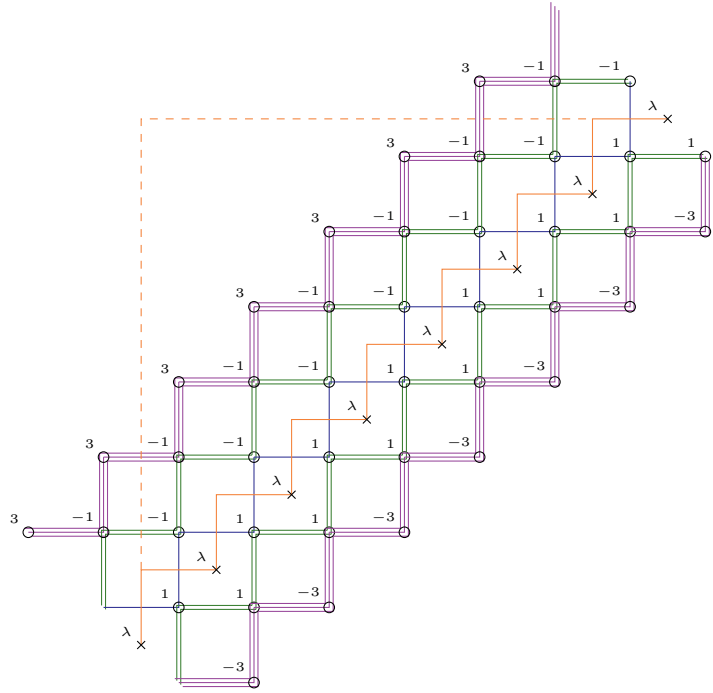
\end{Example}

\section{The center of $\CA(\mathscr{L})$}
\label{sec:center}
In this section we prove Theorem \ref{thm:B}(iii) and Theorem \ref{thm:C}.
Recall the definition of five-vertex configuration in Definition \ref{def:five-vertex}. For an example of a five-vertex configuration, see Figure \ref{fig:43-five-vertex}.
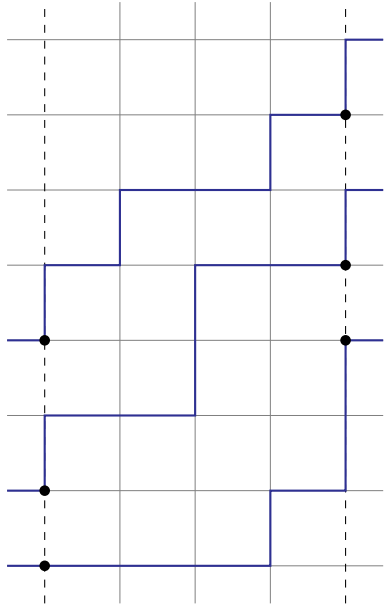
\begin{figure}
\centering
\[
\begin{tikzpicture}
\foreach \x in {1,...,3}{
\draw[help lines] (\x,-.5)--(\x,7.5);
}
\foreach \y in {0,...,7}{
\draw[help lines] (-.5,\y)--(4.5,\y);
}
\draw[dashed] (0,-.5)--(0,7.5);
\draw[dashed] (4,-.5)--(4,7.5);
\draw[thick,Blue] (-.5,0)--(0,0)--(3,0)--(3,1)--(4,1)--(4,3)--(4.5,3);
\draw[thick,Blue] (-.5,1)--(0,1)--(0,2)--(1,2)--(2,2)--(2,3)--(2,4)--(3,4)--(4,4)--(4,5)--(4.5,5);
\draw[thick,Blue] (-.5,3)--(0,3)--(0,4)--(1,4)--(1,5)--(2,5)--(3,5)--(3,6)--(4,6)--(4,7)--(4.5,7);
\fill (0,0) circle (2pt);
\fill (0,1) circle (2pt);
\fill (0,3) circle (2pt);
\fill (4,3) circle (2pt);
\fill (4,4) circle (2pt);
\fill (4,6) circle (2pt);
\end{tikzpicture}
\]
\caption{A $(4,3)$-periodic five-vertex configuration $\mathscr{L}$. By Theorem \ref{thm:C}, the center of the corresponding noncommutative Kleinian fiber product $\CA(\mathscr{L})$ is a Laurent polynomial algebra in one variable.}
\label{fig:43-five-vertex}
\end{figure}

\begin{Definition}
In a higher spin vertex configuration $\mathscr{L}$, a vertex $v\in V$ is \emph{north-eastern (NE)} if there is an edge to the left and below. That is, if $\mathscr{L}_1(v-\al_2/2)>0$ and $\mathscr{L}_2(v-\al_1/2)>0$. Similarly $v$ is \emph{south-western (SW)} if $\mathscr{L}_1(v+\al_2/2)>0$ and $\mathscr{L}_2(v+\al_1/2)>0$. 
\end{Definition}

Let $\Pi(\mathscr{L})$ be the configuration $\mathscr{L}$ regarded as a multiset of non-crossing vertex paths of period $(m,n)$.
Two paths are said to intersect if they share a vertex.

The following is a characterization of five-vertex configurations.
\begin{Lemma}\label{lem:five-characterization}
Let $(m,n)$ be a pair of relatively prime non-negative integers and $\mathscr{L}$ an $(m,n)$-periodic higher spin vertex configuration. The following statements are equivalent.
\begin{enumerate}[{\rm (i)}]
\item \label{it:five} $\mathscr{L}$ is a five-vertex configuration;
\item \label{it:five-NE} In $\mathscr{L}$ there are no NE vertices;
\item \label{it:five-SW} In $\mathscr{L}$ there are no SW vertices;
\item \label{it:five-plus} $P_1^\mathscr{L}(u+\al_2/2)$ and $P_2^\mathscr{L}(u+\al_1/2)$ are relatively prime in $\C[u]$;
\item \label{it:five-minus} $P_1^\mathscr{L}(u-\al_2/2)$ and $P_2^\mathscr{L}(u-\al_1/2)$ are relatively prime in $\C[u]$;
\item \label{it:five-Dyck} Any two distinct paths in $\Pi(\mathscr{L})$ are non-intersecting;
\item \label{it:five-component} Every connected component of $\T_{m,n}\setminus\overline{\mathscr{L}}$ is incontractible;
\item \label{it:five-Modules} There are no simple weight $\CA(\mathscr{L})$-modules of the form $M(D,0)$.
\end{enumerate}
\end{Lemma}
\begin{proof}
\noindent\eqref{it:five}$\Leftrightarrow$\eqref{it:five-NE}$\wedge$\eqref{it:five-SW}:
The local vertex configuration at a vertex $v\in V$ (Figure \ref{fig:vertex}) is one of the five from Figure \ref{fig:five-vertex} if and only if $v$ is neither an NE vertex nor an SW vertex. 
 
\noindent\eqref{it:five-NE}$\Leftrightarrow$\eqref{it:five-plus}:
$v\in V$ is a common root of the polynomials if and only if $v$ is an NE vertex.

\noindent\eqref{it:five-SW}$\Leftrightarrow$\eqref{it:five-minus}: Analogous.

\noindent\eqref{it:five-NE}$\Rightarrow$\eqref{it:five-Dyck}:
Suppose two paths in $\Pi(\mathscr{L})$ intersect at a vertex $v\in V$. Since $v$ is not an NE vertex, either $v-\al_1$ or $v-\al_2$ is a shared vertex. Repeating this argument, we conclude that all vertices are shared between the two paths which means they are identical.

\noindent\eqref{it:five-SW}$\Rightarrow$\eqref{it:five-Dyck}: Analogous.

\noindent\eqref{it:five-Dyck}$\Rightarrow$\eqref{it:five-component}: Obvious.

\noindent\eqref{it:five-component} $\Rightarrow$ \eqref{it:five-NE}$\wedge$\eqref{it:five-SW}:
If $\T_{m,n}\setminus\overline{\mathscr{L}}$ has a contractible connected component $D$, then $\mathscr{L}$ contains both an NE vertex (the upper right corner of $D$) and an SE vertex (the lower left corner).

\noindent\eqref{it:five-component}$\Leftrightarrow$\eqref{it:five-Modules}: Obvious.
\end{proof}

Let $\CA(\mathscr{L})_{\mathrm{loc}}=\CA(\mathscr{L})\otimes_{\C[H]}\C(H)$.

\begin{Proposition}
\label{prp:centralizing-element}
Let $(m,n)$ be a pair of relatively prime positive integers, and let $\mathscr{L}$ be an $(m,n)$-periodic lattice configuration.
For $(k,l)\in\N^2$, define $C_{k,l}\in \CA(\mathscr{L})_{\mathrm{loc}}$ by
\begin{equation}\label{eq:Ckl}
C_{k,l}=X(\un{i})\prod_{\la\in F} (H-\la)^{-\ord(\un{i},\la)}
\end{equation}
where $\un{i}=i_1i_2\cdots i_{k+l}\in\mathsf{Seq}_2(k,l)$ is a sequence of $k$ $1$'s and $l$ $2$'s in any order, and $X(\un{i})=X_{i_{k+l}}^+\cdots X_{i_2}^+X_{i_1}^+$.
Then
\begin{enumerate}[{\rm (a)}]
\item $C_{k,l}$ does not depend on $\un{i}$, only on $(k,l)$.
\item The element
\begin{equation}\label{eq:C}
C=C_{m,n}
\end{equation}
belongs to the center of $\CA(\mathscr{L})_\mathrm{loc}$. In particular, $[C,\CA(\mathscr{L})]=0$.
\item For any $\la\in F$ such that $\bar\la$ belongs to an incontractible connected component $D$ of $\T_{m,n}\setminus\overline{\mathscr{L}}$, the subquotient $\CC_\la$ is generated as a $\C$-algebra by the image of $C$ and its inverse.
\item $C$ belongs to the center of $\CA(\mathscr{L})$
if and only if $\mathscr{L}$ is a five-vertex configuration.
\end{enumerate}
\end{Proposition}
\begin{proof}
(a) Follows immediately from the generalized exchange relation \eqref{eq:GER}.

(b)  Since $C$ has degree $(m,n)$ and $m\al_1+n\al_2=0$, it commutes with $H$. We show that $C$ commutes with $X_1^+$. Then by applying a transposition isomorphism from Lemma \ref{lem:fiber-isomorphisms}\eqref{it:fiber-transposition} it follows that $X_2^+$ also commutes with $C$, and applying the involution $\ast$ it follows that $C$ commutes with $X_1^-$ and $X_2^-$. 
We have
\begin{align*}
X_1^+ C &= X(\un{i}1) \prod_{\la\in F} (H-\la)^{-\ord(\un{i},\la)} \\ 
   &= C_{m+1,n} \prod_{\la\in F} (H-\la)^{\ord(\un{i}1,\la)-\ord(\un{i},\la)} \\ 
   &= X(1\un{i})\prod_{\la\in F} (H-\la)^{-\ord(1\un{i},\la)+\ord(\un{i}1,\la)-\ord(\un{i},\la)}\\
   &= X(\un{i})\prod_{\la\in F} (H-\al_1-\la)^{-\ord(1\un{i},\la)+\ord(\un{i}1,\la)-\ord(\un{i},\la)}X_1^+\\
   &= X(\un{i})\prod_{\la\in F} (H-\la)^{-\ord(1\un{i},\la-\al_1)+\ord(\un{i}1,\la-\al_1)-\ord(\un{i},\la-\al_1)}X_1^+
\end{align*}
It remains to be shown that
\[ -\ord(1\un{i},\la-\al_1)+\ord(\un{i}1,\la-\al_1)-\ord(\un{i},\la-\al_1) = -\ord(\un{i},\la), \]
or equivalently that
\begin{equation}
\ord(\un{i},\la)-\ord(\un{i},\la-\al_1)=\ord(1\un{i},\la-\al_1) - \ord(\un{i}1,\la-\al_1).
\end{equation}
Since the order of a path is independent of the choice of base point, $\ord(1\un{i},\la-\al_1)=\ord(\un{i}1,\la)$. Substituting this it is easy to see that the equality holds. Both sides count the same thing, the sum of orders of corners between the lower path $(\un{i},\la)$ and the upper path $(\un{i},\la-\al_1)$.

(c) By Proposition \ref{prp:B_la-J_la-Description}, $\CC_\la$ is a Laurent polynomial algebra $\C[L,L^{-1}]$ where $L$ is some element of degree $(m,n)$. So it suffices to show that the image in $\CC_\la$ of the degree $(m,n)$ element $C$ is nonzero. Since $\bar\la$ belongs to an incontractible component, there exists $\un{i}\in\mathsf{Seq}_2(m,n)$ such that the based face path $\bar\pi(\un{i},\la)$ does not intersect $\overline{\mathscr{L}}$. By Lemma \ref{lem:order}, $\ord(\un{i},\la)=0$. Thus by part (a), $C=X(\un{i})\prod_{\mu\in F}(H-\mu)^{-\ord(\un{i},\mu)}$ has a well-defined and nonzero image in $\CC_\la$ (obtained by setting $H=\la$ in the product).

(d)  Suppose $C$ belongs to the center of $\CA(\mathscr{L})$. Then so does $C^\ast$. Therefore the element $C^\ast\cdot C$ is a central element of degree zero. That is, $C^\ast \cdot C$ is a polynomial $c(H)$ of $H$ which commutes with $X_i^+$ for $i\in\{1,2\}$. But $X_i^+ c(H) = c(H-\al_i) X_i^+$, hence $c(H)$ is $\al_i$-periodic for $i\in\{1,2\}$. Since $(\al_1,\al_2)\neq (0,0)$, this implies that the polynomial $c(H)$ is a constant, say $c\in \C$. Since $\CA(\mathscr{L})$ is a domain, $c\neq 0$, and hence $C$ is invertible. If $\T_{m,n}\setminus\overline{\mathscr{L}}$ had a contractible connected component $D$, it implies by Theorem \ref{thm:B} that there exists a simple weight module $M(D,0)$ such that $C M=0$ which is absurd. Therefore $\mathscr{L}$ is a five-vertex configuration by Lemma \ref{lem:five-characterization}.

Conversely, by part (a), for each sequence $\un{i}\in\mathsf{Seq}_2(m,n)$ we can write
\begin{equation}\label{eq:pf-center-c1}
C = X(\un{i})\frac{1}{f_{\un{i}}},
\end{equation}
where $f_{\un{i}}=\prod_{\la\in F}(H-\la)^{\ord(\un{i},\la)}$. If every connected component of $\T_{m,n}\setminus\overline{\mathscr{L}}$ is incontractible, for each $\la\in F$ there exists some sequence $\un{i}_\la\in\mathsf{Seq}_2(m,n)$ such that the path $\bar\pi(\un{i}_\la,\la)$ bisects the edges in $\mathscr{L}$ into two disjoint configurations. This implies that $\ord(\un{i}_\la,\la)=0$ by Lemma \ref{lem:order}. By definition of $f_{\un{i}}$, we obtain $f_{\un{i}_\la}(\la)\neq 0$.
Thus the set of polynomials $f_{\un{i}}$ has no common zeros.
Therefore the ideal they generate in the polynomial ring $\C[H]$ contains $1$. That is, there exist $g_{\un{i}}\in\C[H]$ such that the Bezout identity holds:
\[1=\sum_{\un{i}} h_{\un{i}} g_{\un{i}}.\]
Multiplying this by $C$ and using \eqref{eq:pf-center-c1} we get
\[C=\sum_{\un{i}} X(\un{i}) g_{\un{i}},\]
which is an element of $\CA(\mathscr{L})$. By part (b) it thus belongs to the center of $\CA(\mathscr{L})$.
\end{proof}

Proposition \ref{prp:centralizing-element}(c) proves Theorem \ref{thm:B}(iii).
We now prove our third and final main theorem.

\begin{proof}[Proof of Theorem \ref{thm:C}]
Put $\CA=\CA(\mathscr{L})$.
Let $z$ be an element from the center of $\CA(\mathscr{L})$. Let $C=C_{m,n}$ be the central element of $\CA_{\mathrm{loc}}=\CA\otimes_{\C[H]}\C(H)$ from Proposition \ref{prp:centralizing-element}(b). It suffices to show that
 $z\in\C[C,C^{-1}]$, because then $z\in\C[C,C^{-1}]\cap\CA$ and $\C[C,C^{-1}]\cap\CA$ equals $\C[C,C^{-1}]$ if $\mathscr{L}$ is a five-vertex configuration and equals $\C$ otherwise, by Proposition \ref{prp:centralizing-element}(d).

Since the center of any graded algebra is a graded subalgebra, we can assume without loss of generality that $z$ is homogeneous with respect to the $\Z^2$-gradation on $\CA(\mathscr{L})$. Let $(d_1,d_2)=\deg z$. Then $0=[H,z]=(d_1\al_1+d_2\al_2)z$ and hence $(d_1,d_2)=(km,kn)$ for some $k\in\Z$. Replacing $z$ by $z^\ast$ if necessary, we can assume $k\ge 0$.
By Theorem \ref{thm:PIDCGR}, $\CA$ is a crystalline graded ring. Hence, in particular, there exists an element $b_k\in\CA$ of degree $(km,kn)$ such that $\CA_{(km,kn)}$ is a free left $\C[H]$-module with basis element $b_k$. Therefore there exists $z_k\in \C[H]$ such that $z=z_k b_k$.
By Proposition \ref{prp:centralizing-element}(b), there is a nonzero central element $C$ of degree $(m,n)$ in $\CA_{\mathrm{loc}}$. Therefore there exists a rational function $c_k\in \C(H)$ such that $C^k=c_k b_k$. Now
\begin{equation}\label{eq:ckzk1}
 c_k z_k b_k X_i^+ = c_k z X_i^+ = c_k X_i^+ z = c_k X_i^+ z_k b_k = c_k \si_i(z_k) X_i^+ b_k
\end{equation}
where $\si_i$ is the automorphism of $\C(H)$ determined by $\si_i(H)=H-\al_i$.
On the other hand,
\begin{equation}\label{eq:ckzk2}
 c_k z_k b_k X_i^+ = z_k C^k X_i^+ = z_k X_i^+ C^k = z_k X_i^+ c_k b_k = z_k \si_i(c_k) X_i^+ b_k
\end{equation}
Since $\CA_\mathrm{loc}$ is a domain, \eqref{eq:ckzk1}-\eqref{eq:ckzk2} imply that the rational function $z_k/c_k$ is $\al_i$-periodic for $i=1,2$. Since $(\al_1,\al_2)\neq (0,0)$, this implies that $z_k/c_k\in\C$. Hence $z\in \C C^k$ which proves the claim.
\end{proof}

\begin{Example}
In this example we calculate the generator of the center of $\CA(\mathscr{L})$, where $\mathscr{L}$ is given by Figure \ref{fig:21-single}.
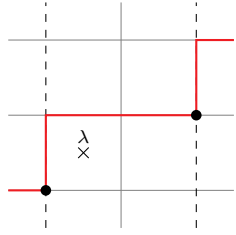
\begin{figure}
\centering
\begin{tikzpicture}[xscale=1,yscale=1]
% Help lines
\draw[help lines] (-.5,0)--(2.5,0);
\draw[help lines] (-.5,1)--(2.5,1);
\draw[help lines] (-.5,2)--(2.5,2);
\draw[help lines] (1,-.5)--(1,2.5);
% Seams
\draw[dashed] (0,-.5)--(0,2.5);
\draw[dashed] (2,-.5)--(2,2.5);
% First path
\draw[-,thick,Red] (-.5,0)--(0,0)--(0,1)--(1,1)--(2,1)--(2,2)--(2.5,2);
% Vertices
\fill (0,0) circle (2pt);
\fill (2,1) circle (2pt);
% Weight spaces
\draw (.5cm-2pt,.5cm-2pt)--(.5cm+2pt,.5cm+2pt);
\draw (.5cm+2pt,.5cm-2pt)--(.5cm-2pt,.5cm+2pt);
\draw (.5,.5) node[font=\scriptsize, anchor=south] {$\la$};
\end{tikzpicture}
\caption{A $(2,1)$-periodic higher spin vertex configuration consisting of a single vertex path.}
\label{fig:21-single}
\end{figure}
Choose $(\al_1,\al_2)=(-1,2)$ and $\la=0$. Put $X_i=X_i^+$, $R=\C[H]$, $\CA=\CA(\mathscr{L})$ for brevity.
By Theorem \ref{thm:C}, $Z(\CA)=\C[C,C^{-1}]$ and by Proposition \ref{prp:centralizing-element},
\begin{equation}\label{eq:ex-21-single-C}
C=X_1^2X_2\frac{1}{(H+1)H}=X_1X_2X_1\frac{1}{H(H-1)}=X_2X_1^2\frac{1}{(H-1)(H-2)},
\end{equation}
each expression corresponding to an element of  $\mathsf{Seq}_2(2,1)=\{211,121,112\}$. For example, the only $\mu\in F$ for which $\ord(211,\mu)=1$ are $\mu=\la=0$ and $\mu=\la+\al_1=-1$, which gives the leftmost equality in \eqref{eq:ex-21-single-C}.
To write $C$ as an element of $\CA(\mathscr{L})$ we need to find $f_1,f_2,f_3\in\C[H]$ such that
\begin{equation}\label{eq:ex-21-single}
f_1\cdot (H+1)H+f_2\cdot H(H-1)+f_3\cdot (H-1)(H-2)=1
\end{equation}
One solution is $f_1=f_3=1/2$, $f_2=-1$.
Multiplying both sides of \eqref{eq:ex-21-single} by $C$, using \eqref{eq:ex-21-single-C} we get
\[C = \frac{1}{2} (X_1^2X_2 -2 X_1X_2X_1 + X_2X_1^2)=\frac{1}{2}[X_1,[X_1,X_2]].\]
The inverse of $C$ is a nonzero scalar multiplied by its dual
\[C^\ast = \frac{1}{2}(Y_2Y_1^2 - 2 Y_1Y_2Y_1 + Y_1^2Y_2) = \frac{1}{2}[[Y_2,Y_1],Y_1], \]
where $Y_i=X_i^-$. That is, $C\cdot C^\ast = C^\ast \cdot C$ is some nonzero complex number.
Note that by \eqref{eq:ex-21-single-C}, since $\CA_{(2,1)}$ is generated as a left $R$-module by the three monomials $X_1^2X_2$, $X_1X_2X_1$, $X_2X_1^2$, it is clear that $\CA_{(2,1)}=R\cdot C = C\cdot R$. This again illustrates that $\CA$ is a crystalline graded ring (see Theorem \ref{thm:PIDCGR}).
\end{Example}

\section{Conclusion and further directions}\label{sec:conclusion}

To each $(m,n)$-periodic higher spin vertex configuration $\mathscr{L}$ we have attached an associative algebra $\CA(\mathscr{L})$. On the one hand, this algebra is a noncommutative deformation of a fiber product of two Kleinian singularities of types $A_{n-1}$ and $A_{m-1}$, where $n$ (respectively $m$) is the number of vertical (respectively horizontal) edges in $\mathscr{L}$, counted with multiplicity.
On the other hand $\CA(\mathscr{L})$ is an example of a twisted generalized Weyl algebra, as well as a crystalline graded ring.

Properties of $\mathscr{L}$ are reflected in $\CA(\mathscr{L})$ and its modules. For example:

 The center of $\CA(\mathscr{L})$ is non-trivial if and only if $\mathscr{L}$ is a five-vertex configuration, in which case the center is a Laurent polynomial algebra in a generator given in terms of $\mathscr{L}$ (Theorem \ref{thm:C}).

For $\la\in\Z$, the left $\CA(\mathscr{L})$-module $\CA(\mathscr{L})/\CA(\mathscr{L})(H-\la)$ has a basis parametrized by lattice paths starting at $\bar\la$ and crossing a minimal number of edges in $\mathscr{L}$ (Corollary \ref{cor:quotient-basis}).

Call $\mathscr{L}$ \emph{non-percolating} if it is a connected subset of $\R^2$ (when viewed as the union of its edges regarded as closed line segments). Then the category of finite-dimensional $\CA(\mathscr{L})$-modules has finitely many isoclasses of simple objects if and only if $\mathscr{L}$ is non-percolating (Theorem \ref{thm:B}).

$\mathscr{L}$ is a six-vertex configuration if all edge labels are $0$ or $1$. This is equivalent to that the polynomials $P_i^\mathscr{L}(u)$ have no multiple roots. This in turn is expected to be related to the global dimension of $\CA(\mathscr{L})$ (due to the Kleinian singularity case  \cite[Thm.~5]{Bav1992}) and to the semisimplicity of the category of finite-dimensional modules over $\CA(\mathscr{L})$ (see \cite[Thm.~3.3]{Bav1992}).

We end by stating some possible future directions of research.

Defining a probability measure on the set $\Omega_{m,n}$ of all configurations $\mathscr{L}$ (for example through local Boltzmann weights), we obtain a \emph{random category} $\mathscr{W}_{m,n}$: the set of integral weight module categories $\big\{\mathscr{W}(\mathscr{L})_\Z\big\}_{\mathscr{L}\in\Omega_{m,n}}$ equipped with a probability measure. 

\begin{Problem}
If local Boltzmann weights are chosen so that the $(m,n)$-periodic higher spin vertex model is integrable (giving explicit formula for the partition function), what does this imply for the random category $\mathscr{W}_{m,n}$?
\end{Problem}

\begin{Problem}
What is the probability that the category of integral weight $\CA(\mathscr{L})$-modules is semi-simple?
\end{Problem}

A quantum analog $\CA_q(\mathscr{L})$ can naturally be defined by replacing $\C[H]$ by $\C[K,K^{-1}]$, any factor $(H-a)$ by $\frac{q^{-a}K-q^aK^{-1}}{q-q^{-1}}$, and suitably $q$-deformed defining relations.
\begin{Problem}
Determine all simple finite-dimensional (or more generally, weight) modules over $\CA_q(\mathscr{L})$ when $q$ is a root of unity.
\end{Problem}

\begin{Problem}
Crawley-Bovey and Holland \cite{CraHol1998} have defined noncommutative deformations of any type $ADE$ Kleinian singularity.
Is there a natural deformation of a fiber product of two such Kleinian singularities that would generalize our type $A\times A$ algebra $\CA_{\al_1,\al_2}(p_1,p_2)$?
\end{Problem}

\begin{Problem}
Since $\CA_{\al_1,\al_2}(p_1,p_2)$ are deformations of three-dimensional singular varieties, are they related to  noncommutative crepant resolutions \cite{Van2004}?
\end{Problem}

\section*{Appendix A}
Here we give some details on the proof of the existence of the surjective homomorphism $\mathcal{W}\to\CA^{(d)}/(C-\la)$ from Example \ref{ex:finite-W}.

Put $(\ad a)(b)=[a,b]=ab-ba$. First we show that if $\al_1\neq 0$ and $p_2(u)$ has degree $k$, then the Serre relation
\begin{equation}\label{eq:w-pf-id1}
(\ad X_1^+)^{k+1}(X_2^+)=0
\end{equation}
holds in $\CA_{\al_1,\al_2}(p_1,p_2)$. To see this we use techniques from \cite{Har2009}. In $\tilde{\CA}$ we have,  using \eqref{eq:Aalbepq-rels},
\begin{equation}\label{eq:w-pf-id2}
X_2^- \cdot (\ad X_1^+)^{k+1}(X_2^+) = X_2^+ \cdot (\Id-\si_1)^{k+1}\big(p_2(H+\tfrac{\al_2}{2})\big)\cdot  X_1^{k+1},
\end{equation}
where $\si_1(H)=H-\al_1$. Since each application of the difference operator $\Id-\si_i$ lowers the degree by one, and $\deg p_2=k$, it follows that the right hand side of \eqref{eq:w-pf-id2} is zero. Thus the left hand side of  \eqref{eq:w-pf-id1} is zero after multiplying from the left by the nonzero polynomial $f(H)=P_2(H-\al_2/2)=X_2^+X_2^-$, and hence belongs to the ideal $\CI$ from Definition \ref{def:NCKFP}. Thus \eqref{eq:w-pf-id1} holds in the quotient $\CA$.

Thus, returning to Example \ref{ex:finite-W}, we immediately conclude that in $\CA^{(d)}$ we have
\[[\varphi(J^+),\varphi(S^+)]=\frac{-1}{4}(\ad X_1^+)^3(X_2^+)=0\]
simply because $p_2$ has degree two.

By results in Section \ref{sec:center}, (here it is crucial that $d>1$ in order for $\mathscr{L}$ to be a five-vertex configuration), the center of $\CA^{(d)}$ is a Laurent polynomial algebra in the degree $(1,1)$ generator
\begin{equation}\label{eq:w-pf-C}
C=X_2^+X_1^+\frac{1}{(H-1)(H-1-d)}=X_1^+X_2^+\frac{1}{H(H-d)}
\end{equation}
Putting $f(H)=(H-1)(H-1-d)$ and $g(H)=H(H-d)$ one verifies the Bezout identity
\[d^2-1=(-2H+d-1)f(H) + (2H-d-3)g(H).\]
Multiplying both sides by $\frac{1}{d^2-1}C$ using \eqref{eq:w-pf-C} we obtain \eqref{eq:w-ex-C}.
To compute $C^\ast\cdot C$, it suffices to simplify it to a an expression involving polynomials in $H$, then set $H=0$ (since we know by the proof of Theorem \ref{thm:C} that $C^\ast \cdot C$ is a non-zero scalar).

By a direct computation we have $\varphi(c_2)=\frac{1}{2}(d^2-1)$.

Next we show that $[S^+,S^-]=(w_2-c_2)J^0$ is preserved by $\varphi$. We have, putting $X_i=X_i^+$,
\[[\varphi(S^+),\varphi(S^-)]=\frac{1}{4}(-X_1^2X_2^2+2(X_1X_2)^2-2(X_2X_1)^2+X_2^2X_1^2)\]
Using \eqref{eq:w-pf-C}, substitute $X_1X_2$ and $X_2X_1$ by a corresponding expression involving $C$ and $H$, twice, to get
\[C^2 (-H+\frac{d+1}{2})=C^2 \varphi(J_0) = \big(\varphi(w_2)-\varphi(c_2)\big) \varphi(J_0)\]
as desired.
It is straightforward to check that the remaining relations in $\mathcal{W}$ are preserved by $\varphi$. This gives the homomorphism $\varphi:\mathcal{W}\to\CA^{(d)}$.

It remains to show that composing $\varphi$ with the canonical map $\CA^{(d)}\to\CA^{(d)}/(C-\la)$ yields a  surjection. By definition, the image of contains $X_1^+, X_1^-, X_2^+$ and $H$. We must show that the image also contains $X_2^-$. We have
\[C^\ast = \frac{1}{(H-1)(H-1-d)}X_1^-X_2^-\]
hence
\[ X_1^+ C^\ast = \frac{1}{H(H-d)}X_1^+X_1^-X_2^-=X_2^-\]
Since $C^\ast \cdot C = 1$, $C^\ast$ is a nonzero scalar in $\CA^{(d)}/(C-\la)$, and thus $X_2^{-}$ is in the image.

\bibliographystyle{siam}

\end{document}